\documentclass[12pt, reqno]{amsart}
\usepackage{amsmath,amssymb,amsthm}
\usepackage{amsmath, amssymb}
\usepackage{amsthm, amsfonts, mathrsfs}
\usepackage{mathptmx}
\usepackage{fullpage}
\usepackage{amsfonts,graphicx}
\numberwithin{equation}{section}
\usepackage[colorlinks=true, pdfstartview=FitV, linkcolor=blue, citecolor=blue, urlcolor=blue]{hyperref}

\allowdisplaybreaks

\numberwithin{equation}{section}

\numberwithin{equation}{section}
\newtheorem{theorem}{Theorem}[section]
\newtheorem{proposition}[theorem]{Proposition}
\newtheorem{lemma}[theorem]{Lemma}
\newtheorem{corollary}[theorem]{Corollary}

\numberwithin{equation}{section}

\theoremstyle{remark}
\newtheorem{remark}[theorem]{Remark}
\def\T{ \mathbb{T} }

\def\T{ \mathbb{T} }

\begin{document}

\title[3D Compressible full MHD equations]{Stabilization by a background magnetic field: global well-posedness of the full compressible viscous non-resistive MHD system without heat-conductivity}

\author[L. Qiao]{Liening  Qiao }
\address[L. Qiao]{School of Mathematics and  Statistics, Shandong University of
Technology,  Zibo 255049,  Shandong Province, China} \email{lnqiao2026@163.com}

\author[J. Sun]{Juntao  Sun}
\address[J. Sun]{ School of Mathematics and  Statistics, Shandong University of
Technology,  Zibo 255049,  Shandong Province, China} \email{jtsun@sdut.edu.cn}

\author[J. Wu]{Jiahong Wu}
\address[J. Wu]{Department of Mathematics, University of Notre
	Dame, Notre Dame, IN 46556, USA } \email{jwu29@nd.edu}

 \author[F. Xu]{Fuyi  Xu$^{\dag}$}
\address[F. Xu]{$^{\dag}$ School of Mathematics and  Statistics, Shandong University of
Technology,  Zibo 255049,  Shandong Province, China} \email{zbxufuyi@163.com (Corresponding author)}


\date{}


\begin{abstract}
	We consider the three-dimensional full compressible magnetohydrodynamic
	(MHD) system on the periodic torus $\mathbb T^3$ in the regime where the
	only dissipative mechanism acting on the system is the viscosity of the
	fluid: the magnetic field is non-resistive and the flow is
	non-heat-conducting. We prove that this system admits a unique global
	smooth solution, together with explicit algebraic decay rates, provided
	that the perturbation $(\mathbf u_0,\,P_0-\bar P,\,\mathbf H_0-\mathbf n)$
	of the equilibrium state $(\mathbf 0,\bar P,\mathbf n)$ is sufficiently
	small in a high-order Sobolev space and the background magnetic field
	$\mathbf n\in\mathbb R^3$ satisfies a Diophantine condition. No smallness
	whatsoever is imposed on the initial density: it is only required to be
	bounded away from vacuum and from infinity, and may exhibit arbitrarily
	large variations. The proof uncovers a hidden dissipation mechanism.
	Although neither the density, nor the pressure, nor the magnetic field is
	endowed with any diffusion or damping of its own, the coupling of these
	quantities with the velocity through the background field $\mathbf n$,
	combined with a Poincar\'e-type inequality of Diophantine origin,
	generates effective dissipation for both the pressure and the magnetic
	field perturbations. The large variations of the density are handled by a
	two-tier energy argument, in which weighted time-decay estimates for the
	intermediate-order energy compensate exactly for the linear-in-time
	growth of the highest-order norm of the density.
\end{abstract}

\maketitle

\tableofcontents

\section{Introduction and main result}
\label{sec1}

\subsection{The model}
Electrically conducting compressible fluids, such as plasmas, liquid
metals and ionised gases, are described by the magnetohydrodynamic (MHD)
equations, which couple the compressible Navier--Stokes equations with
Maxwell's equations of electromagnetism through the Lorentz force. These
equations play a central role in geophysics, astrophysics, cosmology and
engineering; we refer to \cite{L-L,L-L-H} for the physical background. In
three space dimensions the full compressible viscous MHD system reads
\begin{align} \label{no-heat-mhd1}
	\left\{
	\begin{aligned}
		&\partial_{t}\rho+\mathrm{div}(\rho\mathbf u)=0,\\
		&\partial_{t}(\rho\mathbf u)+\mathrm{div}(\rho\,\mathbf u\otimes\mathbf u)
		-\mu\Delta \mathbf u-(\mu+\lambda)\nabla\mathrm{div}\mathbf u
		+\nabla\Big(P+\tfrac12|\mathbf H|^2\Big)=(\mathbf H\cdot\nabla)\mathbf H, \\
		&c_{v}\big(\partial_{t}(\rho\theta)+\mathrm{div}(\rho\theta\mathbf u)\big)
		+P\,\mathrm{div}\mathbf u-\kappa\Delta\theta
		=\mathcal Q(\mathbf u)+\nu|\nabla\times\mathbf H|^2, \\
		&\partial_t\mathbf H+(\mathbf u\cdot \nabla)\mathbf H-\nu\Delta\mathbf H
		=(\mathbf H\cdot \nabla)\mathbf u-\mathbf H\,\mathrm{div}\mathbf u,\\
		&\mathrm{div}\mathbf H=0,
	\end{aligned}
	\right.
\end{align}
where $t>0$, $x\in\mathbb T^3:=[-\pi,\pi]^3$, and the unknowns
$\rho=\rho(t,x)>0$, $\mathbf u=\mathbf u(t,x)$, $\theta=\theta(t,x)>0$ and
$\mathbf H=\mathbf H(t,x)$ denote the density, the velocity field, the
absolute temperature and the magnetic field, respectively. The positive
constants $c_{v}$, $\kappa$ and $\nu$ stand for the specific heat at
constant volume, the coefficient of heat conduction and the magnetic
diffusivity. The constants $\mu$ and $\lambda$ are the shear viscosity and
the second viscosity coefficient, and are subject to the physical
constraints
$$\mu>0,\qquad 2\mu+3\lambda\geq0.$$
The symbol $\otimes$ denotes the tensor product, and $\mathcal Q(\mathbf u)$
is the viscous dissipation function
$$\mathcal Q(\mathbf u):=\frac{\mu}{2}
\big|\nabla\mathbf u+(\nabla\mathbf u)^{T}\big|^2
+\lambda(\mathrm{div}\,\mathbf u)^2\ \ \geq 0,$$
where $(\nabla\mathbf u)^{T}$ is the transpose of $\nabla\mathbf u$.
Throughout the paper the fluid is a perfect gas, so that the pressure obeys
the equation of state
\begin{equation}\label{state-law}
	P=R\rho\theta,\qquad R>0 .
\end{equation}

The mathematical analysis of \eqref{no-heat-mhd1} is delicate because of
the strong coupling between the fluid motion and the magnetic field, and a
great deal of progress has nevertheless been achieved in recent years;
see, e.g., \cite{Gao-Wu-Xu,KawashimaS,Pu-Guo,Xu-Zhong}. When all three
dissipative mechanisms --- viscosity, heat conduction and magnetic
diffusion --- are present, the available theory runs largely parallel to
that of the full compressible Navier--Stokes equations. The situation
changes drastically as soon as some of the coefficients
$\mu,\lambda,\kappa,\nu$ are switched off: the corresponding smoothing
effect disappears, and the global well-posedness and stability theory
becomes genuinely open. For instance, whether the compressible
non-resistive system ($\nu=0$) possesses global small solutions on
$\mathbb R^3$, without any further structural assumption, is still
unknown.

\subsection{The problem under consideration, and why we use the pressure}
In this paper, we treat the extreme case in which \emph{the velocity
	carries the only dissipation in the system}, namely
\begin{equation}\label{coefficients}
	\mu>0,\qquad \mu+\lambda>0,\qquad \kappa=0,\qquad \nu=0 .
\end{equation}
Physically, this is the regime of a viscous, ideally conducting and
thermally insulated compressible fluid. Under \eqref{coefficients} both
the Joule heating $\nu|\nabla\times\mathbf H|^2$ and the heat flux
$-\kappa\Delta\theta$ disappear from the energy balance, and the magnetic
diffusion disappears from the induction equation, so that
\eqref{no-heat-mhd1} becomes
\begin{align} \label{no-heat-mhd1b}
	\left\{
	\begin{aligned}
		&\partial_{t}\rho+\mathrm{div}(\rho\mathbf u)=0,\\
		&\partial_{t}(\rho\mathbf u)+\mathrm{div}(\rho\,\mathbf u\otimes\mathbf u)
		-\mu\Delta \mathbf u-(\mu+\lambda)\nabla\mathrm{div}\mathbf u
		+\nabla\Big(P+\tfrac12|\mathbf H|^2\Big)=(\mathbf H\cdot\nabla)\mathbf H, \\
		&c_{v}\big(\partial_{t}(\rho\theta)+\mathrm{div}(\rho\theta\mathbf u)\big)
		+P\,\mathrm{div}\mathbf u=\mathcal Q(\mathbf u), \\
		&\partial_t\mathbf H+(\mathbf u\cdot \nabla)\mathbf H
		=(\mathbf H\cdot \nabla)\mathbf u-\mathbf H\,\mathrm{div}\mathbf u,\\
		&\mathrm{div}\mathbf H=0 .
	\end{aligned}
	\right.
\end{align}
Three of the four evolution equations in \eqref{no-heat-mhd1b} are now of
pure transport type. This is the source of all the difficulties addressed
below.

\medskip
\noindent\textbf{From the temperature to the pressure.}
Rather than working with the unknowns $(\rho,\mathbf u,\theta,\mathbf H)$,
we shall use the pressure $P$ in place of the temperature $\theta$. Set
\begin{equation}\label{gamma-def}
	\gamma:=\frac{R}{c_{v}}+1>1 .
\end{equation}
Since $\rho\theta=P/R$ by \eqref{state-law}, dividing
$\eqref{no-heat-mhd1b}_3$ by $c_{v}$ and using
$\eqref{no-heat-mhd1b}_1$ turns the energy equation into a closed
equation for $P$,
\begin{equation}\label{pressure-eq}
	\partial_{t}P+(\mathbf u\cdot\nabla)P+\gamma P\,\mathrm{div}\mathbf u
	=(\gamma-1)\,\mathcal Q(\mathbf u).
\end{equation}
Writing the momentum equation in non-conservative form, the system
\eqref{no-heat-mhd1b} is thus equivalent, for solutions bounded away from
vacuum, to
\begin{align} \label{no-heat-mhd2}
	\left\{
	\begin{aligned}
		&\partial_{t}\rho+\mathrm{div}(\rho\mathbf u)=0,\\
		&\rho\big(\partial_{t}\mathbf u+(\mathbf u\cdot\nabla)\mathbf u\big)
		-\mu\Delta \mathbf u-(\mu+\lambda)\nabla\mathrm{div}\mathbf u
		+\nabla\Big(P+\tfrac12|\mathbf H|^2\Big)=(\mathbf H\cdot\nabla)\mathbf H, \\
		&\partial_{t}P+(\mathbf u\cdot\nabla)P+\gamma P\,\mathrm{div}\mathbf u
		=(\gamma-1)\mathcal Q(\mathbf u), \\
		&\partial_t\mathbf H+(\mathbf u\cdot \nabla)\mathbf H
		=(\mathbf H\cdot \nabla)\mathbf u-\mathbf H\,\mathrm{div}\mathbf u,\\
		&\mathrm{div}\mathbf H=0,
	\end{aligned}
	\right.
\end{align}
supplemented with the initial data
\begin{equation}\label{initial-data}
	(\rho,\mathbf u,P,\mathbf H)\big|_{t=0}
	=(\rho_0,\mathbf u_0,P_0,\mathbf H_0)\quad\text{on }\quad \mathbb T^3 .
\end{equation}

This change of unknowns is not merely cosmetic; it is what makes the
problem tractable, and we pause to explain why. It is by now a standard
device for the full compressible Navier--Stokes and MHD equations without
heat conductivity (see, e.g., \cite{Duan-Ma,Gao-Tan-Wu,Tan-Xu-Wang}), but
in the present setting it plays an additional and essential role.

\emph{(i) It decouples the density from the rest of the system.}
The temperature equation $\eqref{no-heat-mhd1b}_3$ involves $\rho$ and
$\theta$ only through the product $\rho\theta$, and the momentum equation
sees the thermodynamics only through $\nabla P=R\nabla(\rho\theta)$.
Consequently, in the variables $(\rho,\mathbf u,\theta,\mathbf H)$, every
high-order estimate of the velocity requires controlling the same number
of derivatives of $\rho$ and of $\theta$ \emph{simultaneously}, together
with all the mixed products generated by the Leibniz rule. After the
substitution, the triple $(\mathbf u,P,\mathbf H)$ satisfies the closed
system $\eqref{no-heat-mhd2}_{2,3,4,5}$, in which the density survives
\emph{only} as the scalar coefficient $\rho$ in front of the material
derivative of $\mathbf u$. The continuity equation
$\eqref{no-heat-mhd2}_1$ then decouples and may be solved a posteriori
along the flow of $\mathbf u$.

\emph{(ii) It is what allows arbitrarily large density variations.}
Because $\rho$ obeys a pure transport equation with no dissipation, its
Sobolev norms are not expected to remain bounded, and no smallness of
$\nabla\rho$ can be propagated in time --- indeed, our estimates only give
$\|\rho(t)\|_{H^{N}}^2\lesssim\|\rho_0\|^2_{H^{N}}(1+\epsilon t)$.  Had we
kept $\theta$ as an unknown, smallness of $\nabla P$ would have had to be
extracted from smallness of $\nabla\rho$ and $\nabla\theta$ at the same
time, which is incompatible with a density of arbitrarily large
oscillation. Working with $P$ transfers the whole burden of the pressure
estimate onto a variable which, as we show below, obeys a favourable
damped-wave equation. This is precisely why the initial density in
Theorem \ref{th:main1} need only satisfy
$\underline\rho\le\rho_0\le\bar\rho$, with no bound on $\bar\rho/\underline\rho$
and no smallness of $\|\rho_0\|_{H^M}$.

\emph{(iii) It exposes the acoustic structure.}
Near an equilibrium state, the pressure perturbation is exactly the
acoustic variable: taking the divergence of the momentum equation and
substituting \eqref{pressure-eq} produces a damped acoustic equation,
see \eqref{2026-6-11-3} below, whose linear part is elliptic in the
frequency variable. In the variables $(\rho,\theta)$ the same information
is distributed between two transported quantities and is not directly
usable. For the same reason, the Diophantine mechanism described in
Subsection \ref{subsec-linear} applies to $(\mathbf n\cdot\nabla)P$, a
quantity that can be read off the momentum equation by simply projecting
it onto $\mathbf n$.

\emph{(iv) It produces a clean conservation law.}
The right-hand side $(\gamma-1)\mathcal Q(\mathbf u)$ in
\eqref{pressure-eq} is quadratic in $\nabla\mathbf u$ and is therefore
exactly absorbed by the viscous dissipation; integrating
\eqref{pressure-eq} over $\mathbb T^3$ and combining it with the momentum
and induction equations yields the conservation law
$$\frac{d}{dt}\Big(\|(\sqrt\rho\,\mathbf u,\mathbf H)\|^2_{L^2}
+\frac{2}{\gamma-1}\int_{\mathbb T^3}P\,dx\Big)=0,$$
see Lemma \ref{priori-basic}, which is the starting point of all our
estimates and which fixes, once and for all, the constant reference
pressure $\bar P$ (see Remark \ref{rem-Pbar}).

Let us stress that we impose no smallness on $\rho_0$ and no compatibility
condition of any kind; the temperature is recovered a posteriori from
$\theta=P/(R\rho)$, and remains positive and smooth as long as $P$ and
$\rho$ do.

\subsection{Stabilization by a background magnetic field}
Since the system \eqref{no-heat-mhd2} carries a single dissipative term,
global existence cannot be expected for arbitrary equilibria: the
inviscid compressible flow is governed by the Euler equations, whose
smooth solutions develop singularities in finite time even for small
smooth initial data. The extra stabilizing agent that we exploit is a
\emph{background magnetic field}.

This mechanism is well documented in the physics literature. It was
identified theoretically by Alfv\'en \cite{ALF}, who showed that a uniform
magnetic field supports transverse waves --- now called Alfv\'en waves ---
which transport momentum and energy along the field lines. Gallet,
Berhanu and Mordant \cite{GBM} reported measurements of forced turbulence
in a swirling liquid-sodium flow subject to an externally imposed magnetic
field: as the field strength was increased, the turbulent fluctuations
were progressively suppressed and the flow became quasi-two-dimensional
and laminar along the field direction. Califano and Chiuderi \cite{CC}
demonstrated numerically that a resistivity-independent dissipation of MHD
waves takes place in inhomogeneous plasmas threaded by a background field,
showing that the field geometry alone can transfer energy from large to
small scales in a controlled way. Mathematically, the analysis of the MHD
equations near a nontrivial background magnetic field is an attempt to
capture this stabilizing phenomenon rigorously.

On the torus, the relevant algebraic obstruction is a small-divisor
problem. The directional derivative $\mathbf n\cdot\nabla$ acts on the
Fourier mode $e^{i\mathbf k\cdot x}$ as multiplication by
$i(\mathbf n\cdot\mathbf k)$, so that inverting it with a controlled loss
of derivatives requires the numbers $\mathbf n\cdot\mathbf k$ to be
quantitatively bounded away from zero for all
$\mathbf k\in\mathbb Z^3\setminus\{0\}$. This is exactly the Diophantine
condition \eqref{Diophantine0} below, which holds for Lebesgue-almost
every $\mathbf n\in\mathbb R^3$ and which we shall assume throughout.

\subsection{Known results}
We now review the results that are most directly related to ours,
organised according to which dissipative mechanisms are retained.

\medskip
\noindent\textbf{The isentropic case.} When the energy equation is
replaced by a barotropic pressure law, the system with full viscosity
($\mu>0$, $\mu+\lambda>0$) and resistivity ($\nu>0$) has been extensively
studied; see, e.g.,
\cite{Gao-Wu-Xu,Hu-Wang2,KawashimaS,L-X-Z,L-M-X,Suen,SuenA} and the
references therein. The problem becomes substantially harder when only one
of the two mechanisms is present.

For the viscous non-resistive case ($\mu>0,\ \mu+\lambda>0,\ \nu=0$),
Hu \cite{HuX} constructed a global solution of the Cauchy problem for the
2D system with data close to a constant background and with the Lagrangian
deformation gradient near the identity. Wu and Wu \cite{Wu-Wu} developed a
systematic approach to global well-posedness and stability for the 2D
compressible non-resistive system near an equilibrium, and Dong, Wu and
Zhai \cite{Dong-Wu-Zhai} treated the $2\frac12$D case. In a 3D strip
domain and near the background field $(0,0,1)$, Tan and Wang
\cite{Tan-Wang} obtained the global well-posedness and stability for small initial
data satisfying suitable compatibility conditions; Jiang and Jiang
\cite{Jiang-Jiang-1} studied the Rayleigh--Taylor problem for the 3D
stratified compressible system; and Zhao \cite{ZhaoY1} devised an approach
requiring no compatibility conditions on the initial  data. On periodic domains,
Wu and Zhu \cite{Wu-Zhu} established the global well-posedness and stability
for the 2D system near the background field $(0,1)$. Following the
incompressible works \cite{Chen-Zhang-Zhou,Xie-Jiu-Liu}, Wu and Zhai
\cite{Wu-Zhai} obtained the global smooth solutions on $\mathbb T^3$ for the  initial data
close to a background magnetic field satisfying a Diophantine condition,
and Jiu, Liu and Xie \cite{Jiu-Liu-Xie} lowered the regularity
requirement. See also \cite{Jiang-Jang2,Zhai-Wu-Xu,ZhaoY2}. For the
inviscid resistive case ($\mu=\lambda=0$, $\nu>0$), Li and Qiao
\cite{Li-Qiao} proved global well-posedness and stability on the
three-dimensional torus.

\medskip
\noindent\textbf{The full (non-isentropic) case.} Here the interaction
between the hydrodynamic and the electrodynamic effects generates
additional obstructions, and many basic questions remain open --- for
instance, the global existence and uniqueness of classical solutions with
large initial  data is unresolved even in one space dimension for densities bounded
away from vacuum. Hu and Wang \cite{Hu-Wang1} constructed the global weak
solutions with large initial  data and described their large-time behaviour, while
the local existence of strong solutions for general initial  data was obtained in
\cite{FY,XH}.

When the viscosity and the heat conduction are retained but the
resistivity is switched off ($\mu>0,\ \mu+\lambda>0,\ \kappa>0,\ \nu=0$),
Li \cite{LiY} extended the result of Tan and Wang \cite{Tan-Wang} to the
non-isentropic setting, and Zhao \cite{ZhaoY3} removed the compatibility
conditions on the initial  data. On the torus, Li, Xu and Zhai \cite{Li-Xu-Zhai}
proved global well-posedness and stability for small initial  data near a
background magnetic field satisfying the Diophantine condition; see also
\cite{Li-Sun,Yang-Zhao,Zhai-Li-Zhao,Zhai-Zhang}.

When instead the heat conduction is switched off but the resistivity is
retained ($\mu>0,\ \mu+\lambda>0,\ \nu>0,\ \kappa=0$), Gao, Tan and Wu
\cite{Gao-Tan-Wu} obtained the global existence and convergence rates in
$H^3(\mathbb R^3)$ under the assumptions of \cite{Duan-Ma}; Tan, Xu and
Wang \cite{Tan-Xu-Wang} removed the $L^1(\mathbb R^3)$ bound on the initial  data
in an $H^l$ framework with $l>3$; Wu, Tan and Zou \cite{Wu-Tan-Zou}
worked in $H^2$; and Liang \cite{LiangJ} recently allowed large
perturbations of the initial density in $\mathbb R^3$. See also
\cite{Fan-Wang-Zhou,WangY,Wei-Li-Guo}.

Finally, when the heat conduction and the resistivity are present but the
viscosity is absent ($\mu=\lambda=0$, $\kappa>0$, $\nu>0$), the underlying
inviscid flow may form shocks, and results are scarce. Wang and Xin
\cite{Wang-Xin} proved the  global well-posedness for small initial  data near a
non-horizontal constant background field in a 3D strip domain, and Wu, Xu
and Zhai \cite{Wu-Xu-Zhai} treated the periodic case under the
Diophantine condition.

\medskip
\noindent\textbf{One dissipation only.} In all of the works quoted above,
at least two of the three dissipative mechanisms are active. When a single
one survives, the situation changes qualitatively. Qiao, Wu, Xu and Zhai
\cite{Qiao-Zhai} studied the 3D compressible isentropic \emph{ideal}
($\mu=\lambda=\nu=0$) MHD system with velocity damping and constructed a
unique global smooth solution for data close to an equilibrium, the
background field being subject to the Diophantine condition; Lin and the
fourth author \cite{Lin-Xu} extended this to the non-isentropic case
\emph{with} heat conduction. To the best of our knowledge, the case
$\kappa=\nu=0$ of the full system \eqref{no-heat-mhd1}, in which the
viscosity of the fluid is the only dissipative mechanism at hand, has
remained entirely open. This is the case we settle here.

\subsection{Notation}
We write $\Lambda:=\sqrt{-\Delta}$, and $f\lesssim g$ for
$f\leq Cg$, and $f\sim g$ for $C^{-1}g\leq f\leq Cg$, where $C>0$ is a
generic constant that may change from line to line and depends only on
$\mathbb T^3$, $r$, $c$, $|\mathbf n|$, $N$, $\mu$, $\lambda$, $\gamma$,
$\bar P$, $\underline{\rho}^{-1}$, $\bar\rho$, $M_0$. We denote by $L^q(\mathbb T^3)$,
$q\geq1$, and $H^k(\mathbb T^3)$, $k\geq0$, the usual Lebesgue and Sobolev
spaces on $\mathbb T^3$, with norms $\|\cdot\|_{L^q}$ and $\|\cdot\|_{k}$;
in particular $\|\cdot\|_0=\|\cdot\|_{L^2}$. Unless otherwise stated,
$\int\cdot\,dx$ means $\int_{\mathbb T^3}\cdot\,dx$. We abbreviate
$\|h(f,g)\|_X=\|hf,hg\|_X:=(\|hf\|_X^2+\|hg\|^2_X)^{1/2}$, and we write
$[U,V]W:=U(VW)-V(UW)$ for the commutator.

\subsection{Main result}
A vector $\mathbf n\in\mathbb R^3$ is said to satisfy the
\emph{Diophantine condition} if
\begin{equation}\label{Diophantine0}
	|\mathbf n\cdot \mathbf{k}|\ge \frac{c}{|\mathbf{k}|^r}
	\quad\text{for all }\mathbf k\in\mathbb Z^3\setminus\{0\},
	\text{ for some } c>0 \text{ and } r>2.
\end{equation}
Vectors with three rational components fail \eqref{Diophantine0}, but
almost every $\mathbf n\in\mathbb R^3$ satisfies it, as recalled in
\cite{Chen-Zhang-Zhou} in the context of the incompressible MHD equations.
Studying the dynamics near a vector field obeying a Diophantine condition
is a familiar practice in ergodic theory and dynamical systems; see, e.g.,
\cite{Cas,Koc,Lop}.

Given a positive constant $\bar P$ and a constant vector $\mathbf n$
satisfying \eqref{Diophantine0}, we set
$$\mathbf h:=\mathbf H-\mathbf n,\qquad p:=P-\bar P,$$
with $\mathbf h_0:=\mathbf H_0-\mathbf n$ and $p_0:=P_0-\bar P$. Then
\eqref{no-heat-mhd2} is equivalent to the perturbed system
\begin{align} \label{no-heat-mhd3}
	\left\{
	\begin{aligned}
		&\partial_{t}\rho+\mathrm{div}(\rho\mathbf u)=0,\\
		&\rho\big(\partial_{t}\mathbf u+(\mathbf u\cdot\nabla)\mathbf u\big)
		-\mu\Delta \mathbf u-(\mu+\lambda)\nabla\mathrm{div}\mathbf u
		+\nabla\Big(p+\tfrac12|\mathbf h|^2+\mathbf h\cdot\mathbf n\Big)
		=(\mathbf h\cdot\nabla)\mathbf h+(\mathbf n\cdot\nabla)\mathbf h, \\
		&\partial_{t}p+(\mathbf u\cdot\nabla)p+\gamma\bar P\,\mathrm{div}\mathbf u
		+\gamma p\,\mathrm{div}\mathbf u=(\gamma-1)\mathcal Q(\mathbf u), \\
		&\partial_t\mathbf h+(\mathbf u\cdot \nabla)\mathbf h
		=(\mathbf h\cdot \nabla)\mathbf u+(\mathbf n\cdot \nabla)\mathbf u
		-\mathbf h\,\mathrm{div}\mathbf u-\mathbf n\,\mathrm{div}\mathbf u,\\
		&\mathrm{div}\mathbf h=0 .
	\end{aligned}
	\right.
\end{align}
Our main result reads as follows.

\begin{theorem}\label{th:main1}
	Let $\mathbf n\in\mathbb R^3\setminus\{0\}$ satisfy the Diophantine
	condition \eqref{Diophantine0}, and let $M,N$ be two  positive integers
	with
	$$M\geq r+1,\qquad N> M+2(r+1).$$
	Set
	$$M_0:=\int_{\mathbb T^3}\rho_0\,dx,\qquad
	M_{\rho}:=\|\rho_0\|^{\frac{4M-4}{2M-5}}_{M}+\|\rho_0\|^{2(M-1)}_{M}+1 .$$
	Assume that the initial data
	$(\rho_0,\mathbf u_0,p_0,\mathbf h_0)\in H^{N}(\mathbb T^3)$ satisfy
	\begin{equation}\label{xiajie-1}
		\underline{\rho}\leq\rho_0\leq\bar\rho \quad\text{on }\quad \mathbb T^3
	\end{equation}
	for two positive constants $\underline\rho,\bar\rho$, together with
	\begin{equation}\label{junzhi-1}
		\int\rho_0\mathbf u_0\,dx=\int\mathbf h_0\,dx=0,
	\end{equation}
	and
	\begin{equation}\label{chuzhixiaoyu0}
		\frac{\gamma-1}{2}\|(\sqrt{\rho_0}\,\mathbf u_0,\mathbf h_0)\|^2_0
		+\int p_0\,dx=0 .
	\end{equation}
	Then there is a constant $\epsilon>0$, depending only on
	$\mathbb T^3$, $r$, $c$, $|\mathbf n|$, $M$, $N$, $\mu$, $\lambda$,
	$\gamma$, $\bar P$, $\underline\rho^{-1}$, $\bar\rho$, $M_0$,
	$\|\rho_0\|_{M}$ and $\|\rho_0\|_{N}$, such that if
	\begin{equation}\label{initial-date-1}
		\|(\mathbf u_0,p_0,\mathbf h_0)\|_{N}\leq\epsilon,
	\end{equation}
	then the system \eqref{no-heat-mhd3}--\eqref{initial-data} admits a unique
	global solution
	$(\rho,\mathbf u,p,\mathbf h)\in C([0,\infty);H^{N}(\mathbb T^3))$
	satisfying, for every $t\in[0,\infty)$,
	\begin{align}
		&\frac{2}{3}\underline{\rho}\leq\rho\leq\frac32\bar\rho
		\quad\text{on }\mathbb T^3\times[0,\infty),
		\label{th-rho-bounds}\\[2pt]
		&\sup_{t\in[0,\infty)}\|\rho\|^2_{M}\leq 2\|\rho_0\|_{M}^2,
		\label{th-rho-M}\\[2pt]
		&\sup_{\tau\in[0,t]}\|\rho\|^2_{N}\lesssim\|\rho_0\|^2_{N}(1+\epsilon t),
		\label{th-rho-N}\\[2pt]
		&\sup_{\tau\in[0,t]}\|(\mathbf u,p,\mathbf h)\|^2_M
		\lesssim\Big(t+\big(M_\rho\|(\mathbf u_0,p_0,\mathbf h_0)\|^2_M
		\big)^{-\frac{r+1}{N-M}}\Big)^{-\frac{N-M}{r+1}},
		\label{th-decay}\\[2pt]
		&\sup_{t\in[0,\infty)}\|(\mathbf u,p,\mathbf h)\|^2_N
		+\int_0^{\infty}\|\mathbf u_t\|^2_{N-1}\,dt
		+\int_0^{\infty}\|\mathbf u\|_{N+1}^2\,dt\leq\epsilon .
		\label{th-highest}
	\end{align}
\end{theorem}

Several comments are in order.

\begin{remark}\label{rem-Pbar}
	If one starts from data $(\mathbf u_0,P_0,\mathbf H_0)\in
	H^{N}(\mathbb T^3)$, the normalisation \eqref{chuzhixiaoyu0} is not a
	restriction: it is achieved by the single admissible choice
	$$\bar P:=|\mathbb T^3|^{-1}\Big(\int P_0\,dx
	+\frac{\gamma-1}{2}\big\|\big(\sqrt{\rho_0}\,\mathbf u_0,
	\mathbf H_0-\mathbf n\big)\big\|^2_0\Big),$$
	which is exactly the constant selected by the conservation law of
	Lemma \ref{priori-basic}.
\end{remark}

\begin{remark}
	No smallness is imposed on $\rho_0$, and the ratio
	$\bar\rho/\underline\rho$ may be arbitrarily large; the density is only
	required to be bounded away from vacuum and from infinity. The price is
	that the two quantities $\|\rho_0\|_M$ and $\|\rho_0\|_N$ enter the
	threshold $\epsilon$: the density may be as large and as rough as one
	likes, but the size of the admissible perturbation
	$(\mathbf u_0,p_0,\mathbf h_0)$ decreases as it grows. The same two
	quantities govern, through the constant $M_\rho$, how the decay rate
	\eqref{th-decay} degrades with the density oscillation.
\end{remark}

\begin{remark}
	Choosing $M=r+1$ gives $N>3(r+1)$, which matches the regularity required
	in the incompressible and isentropic settings
	\cite{Jiu-Liu-Xie,Xie-Jiu-Liu}. Interpolating \eqref{th-decay} with
	\eqref{th-highest} by the Gagliardo--Nirenberg inequality yields the full
	family of intermediate decay rates: for every $L\in[r+1,N]$ and every
	$t\geq0$,
	\begin{align*}
		\sup_{\tau\in[0,t]}\|(\mathbf u,p,\mathbf h)\|^2_L
		&\lesssim\sup_{\tau\in[0,t]}\|(\mathbf u,p,\mathbf h)\|^{\frac{2(N-L)}{N-r-1}}_{r+1}
		\|(\mathbf u,p,\mathbf h)\|^{\frac{2(L-r-1)}{N-r-1}}_{N}\\
		&\lesssim\epsilon^{\frac{L-r-1}{N-r-1}}
		\biggl(t+\Big(\big(1+\|\rho_0\|^{\frac{4r}{2r-3}}_{r+1}
		+\|\rho_0\|^{2r}_{r+1}\big)
		\|(\mathbf u_0,p_0,\mathbf h_0)\|^2_{r+1}\Big)^{-\frac{r+1}{N-r-1}}
		\biggr)^{-\frac{N-L}{r+1}} .
	\end{align*}
	In particular the decay is faster the more regular the data, and it
	degenerates, as it must, at the top regularity $L=N$.
\end{remark}

To summarise, the main contributions of this paper are the following.

\begin{itemize}
	\item[(a)] Theorem \ref{th:main1} appears to be the first global
	well-posedness and stability result for the \emph{full} compressible MHD
	system in which the fluid viscosity is the only dissipative mechanism:
	there is neither heat conduction nor magnetic diffusion, and neither the
	density, the pressure nor the magnetic field enjoys any damping.
	
	\item[(b)] The initial density is allowed to have arbitrarily large
	variations. This is made possible by the pressure formulation
	\eqref{no-heat-mhd2}, in which the density only enters as a coefficient,
	and it forces us to keep track of the variable coefficient $\rho$ in every
	estimate, including at the linear level.
	
	\item[(c)] We obtain explicit algebraic decay rates for the
	intermediate-order energy, and the mechanism producing them is a
	\emph{hidden} dissipation: the pressure and magnetic perturbations are
	dissipated not by any term in their own equations, but through their
	coupling to the velocity across the background field $\mathbf n$.
	
	\item[(d)] At the technical level, we introduce a two-tier energy scheme in
	which weighted time-decay estimates for the intermediate ($M$-th order)
	energy are used to absorb the linear-in-time growth of the highest-order
	density norm \eqref{th-rho-N}. We believe this scheme is of independent
	interest for hyperbolic--parabolic systems with partial dissipation and
	non-small transported coefficients.
\end{itemize}

\subsection{The stabilizing mechanism at the linear level}
\label{subsec-linear}
Before turning to the proof, we explain the mechanism on the linearisation
of \eqref{no-heat-mhd3}. Since the density solves a transport equation,
its linearisation is $\partial_t\tilde\rho=0$, so that
$\tilde\rho\equiv\tilde\rho_0(x)$ is a fixed function of $x$ alone, bounded
above and below by \eqref{xiajie-1}. Normalising
$\mu=1$, $\mu+\lambda=0$ and $\gamma\bar P=1$ for readability, the
linearised system reads
\begin{equation}\label{no-heat-mhd-5-28}
	\left\{
	\begin{aligned}
		&\tilde\rho_0\,\partial_{t}\tilde{\mathbf u}-\Delta\tilde{\mathbf u}
		+\nabla(\tilde{p}+\tilde{\mathbf h}\cdot\mathbf n)
		=(\mathbf n\cdot\nabla)\tilde{\mathbf h}, \\[3pt]
		&\partial_{t}\tilde{p}+\mathrm{div}\,\tilde{\mathbf u}=0, \\[3pt]
		&\partial_t\tilde{\mathbf h}
		=(\mathbf n\cdot \nabla)\tilde{\mathbf u}
		-\mathbf n\,\mathrm{div}\,\tilde{\mathbf u}, \\[3pt]
		&\mathrm{div}\,\tilde{\mathbf h}=0,\\[3pt]
		&(\tilde{\mathbf u},\tilde p,\tilde{\mathbf h})\big|_{t=0}
		=(\tilde{\mathbf u}_0,\tilde p_0,\tilde{\mathbf h}_0).
	\end{aligned}
	\right.
\end{equation}
Observe that we do \emph{not} freeze $\tilde\rho_0$ at a constant value:
retaining it as a variable coefficient is unavoidable here, since we allow
large density variations. As in
\cite{Li-Xu-Zhai,Wu-Zhai,Xie-Jiu-Liu,Zhai-Wu-Xu}, understanding
\eqref{no-heat-mhd-5-28} is the key to the nonlinear problem.

\medskip
\noindent\textbf{Effective dissipation for the velocity.}
Differentiating $\eqref{no-heat-mhd-5-28}_1$ in time and eliminating
$\tilde p$ and $\tilde{\mathbf h}$ by means of
$\eqref{no-heat-mhd-5-28}_{2,3}$, we obtain the following  single equation
\begin{equation}\label{2026-6-12-1}
	\tilde\rho_0\partial_t^2\tilde{\mathbf u}
	-\Delta\partial_t\tilde{\mathbf u}
	-(\mathbf n\cdot\nabla)^2\tilde{\mathbf u}
	-(1+|\mathbf n|^2)\nabla\mathrm{div}\,\tilde{\mathbf u}
	+\mathbf n(\mathbf n\cdot\nabla)\mathrm{div}\,\tilde{\mathbf u}
	+\nabla(\mathbf n\cdot\nabla)(\tilde{\mathbf u}\cdot\mathbf n)=0 .
\end{equation}
The point is that \eqref{2026-6-12-1} can be recast as a degenerate damped
wave equation with a variable coefficient,
\begin{equation}\label{2026-6-11-2}
	\tilde\rho_0\partial_t^2\tilde{\mathbf u}
	-\Delta\partial_t\tilde{\mathbf u}
	-\frac{1}{1+|\mathbf n|^2}(\mathbf n\cdot\nabla)^2\tilde{\mathbf u}
	+\mathbb F^{\top}\mathbb F\,\tilde{\mathbf u}=0,
\end{equation}
where
$$\mathbb F^{\top}\mathbb F\,\tilde{\mathbf u}
:=-\frac{|\mathbf n|^2}{1+|\mathbf n|^2}(\mathbf n\cdot\nabla)^2\tilde{\mathbf u}
+\mathbf n(\mathbf n\cdot\nabla)\mathrm{div}\,\tilde{\mathbf u}
+\nabla(\mathbf n\cdot\nabla)(\tilde{\mathbf u}\cdot\mathbf n)
-(1+|\mathbf n|^2)\nabla\mathrm{div}\,\tilde{\mathbf u}$$
and
\begin{equation*}
\mathbb F:=\frac{|\mathbf n|}{\sqrt{1+|\mathbf n|^2}}
(\mathbf n\cdot\nabla)\,\mathbb I
-\frac{\sqrt{1+|\mathbf n|^2}}{|\mathbf n|}\,\mathbf n\otimes\nabla,
\quad\text{that is,}\quad
\mathbb F\mathbf v=\frac{|\mathbf n|(\mathbf n\cdot\nabla)\mathbf v}{\sqrt{1+|\mathbf n|^2}}
-\frac{\sqrt{1+|\mathbf n|^2}}{|\mathbf n|}\,
\mathbf n\,\mathrm{div}\,\mathbf v .
\end{equation*}
A direct integration by parts shows that
$\int \mathbb F^{\top}\mathbb F\mathbf v\cdot\mathbf v\,dx
=\|\mathbb F\mathbf v\|^2_{0}\geq0$, so that
$\mathbb F^{\top}\mathbb F$ is a nonnegative differential operator and the
splitting in \eqref{2026-6-11-2} is an admissible one: the coefficient
$|\mathbf n|^2/(1+|\mathbf n|^2)$ is the smallest fraction of
$-(\mathbf n\cdot\nabla)^2$ that completes the remaining terms into a
perfect square, by the Cauchy--Schwarz inequality
$(\mathbf n\cdot\mathbf v)^2\le|\mathbf n|^2|\mathbf v|^2$.

Equation \eqref{2026-6-11-2} is the mathematical heart of the matter. The
damping $-\Delta\partial_t\tilde{\mathbf u}$ comes from the viscosity,
which is the only dissipation available. The nonnegative operator
$\mathbb F^{\top}\mathbb F$ and the residual term
$-\frac{1}{1+|\mathbf n|^2}(\mathbf n\cdot\nabla)^2\tilde{\mathbf u}$ are
generated \emph{indirectly}, by the coupling of the velocity with the
pressure and the magnetic field through the background field $\mathbf n$;
neither would be present for $\mathbf n=0$. Since $\mathbf n$ obeys the
Diophantine condition \eqref{Diophantine0}, Lemma
\ref{Diophantine-inequality} allows higher-order derivatives taken along
the direction $\mathbf n$ to control lower-order full derivatives, at the
price of a fixed loss of $r$ derivatives. Consequently, in the higher-order
energy estimates the term
$-\frac{1}{1+|\mathbf n|^2}(\mathbf n\cdot\nabla)^2\tilde{\mathbf u}$
provides genuine control of the lower-order derivatives of
$\tilde{\mathbf u}$, and hence decay for $\tilde{\mathbf u}$ and
$\partial_t\tilde{\mathbf u}$.

\medskip
\noindent\textbf{Effective dissipation for the pressure and the magnetic
	field.}
Taking the divergence of $\eqref{no-heat-mhd-5-28}_1$ and using
$\eqref{no-heat-mhd-5-28}_{2,4}$ gives a damped acoustic equation for the
pressure perturbation,
\begin{align}\label{2026-6-11-3}
	\tilde\rho_0\partial_t^2\tilde p-\Delta\partial_t\tilde p-\Delta\tilde p
	=\Delta(\tilde{\mathbf h}\cdot\mathbf n)
	+\partial_t\tilde{\mathbf u}\cdot\nabla\tilde\rho_0 .
\end{align}
Its left-hand side is a damped wave operator; its right-hand side is a
source driven by the magnetic perturbation and --- this is a feature
specific to non-constant densities --- by $\nabla\tilde\rho_0$. Applying
$\partial_t$ and $\Delta$ to $\eqref{no-heat-mhd-5-28}_3$ and eliminating
$\Delta\tilde{\mathbf u}$ through $\eqref{no-heat-mhd-5-28}_1$ yields, in
turn, a degenerate damped equation for the magnetic perturbation,
\begin{align}\label{2026-6-11-4}
	\partial_t^2\tilde{\mathbf h}-\Delta\partial_t\tilde{\mathbf h}
	-(\mathbf n\cdot\nabla)^2\tilde{\mathbf h}
	=(\mathbf n\cdot\nabla)\big((1-\tilde\rho_0)\partial_t\tilde{\mathbf u}\big)
	-\nabla(\mathbf n\cdot\nabla)(\tilde p+\tilde{\mathbf h}\cdot\mathbf n)
	-\mathbf n\big(\partial_t\mathrm{div}\,\tilde{\mathbf u}
	-\Delta\,\mathrm{div}\,\tilde{\mathbf u}\big).
\end{align}
The operator $(\mathbf n\cdot\nabla)^2$ on the left-hand side governs the
propagation of Alfv\'en waves along $\mathbf n$ at speed $|\mathbf n|$,
while the damping $-\Delta\partial_t\tilde{\mathbf h}$ is again inherited
from the velocity: it is the mathematical embodiment of the physical
stabilization described above.

Together, \eqref{2026-6-11-3} and \eqref{2026-6-11-4} exhibit a double
wave structure: the acoustic wave carried by $\tilde p$ is forced by the
magnetic and velocity perturbations, while the Alfv\'en wave carried by
$\tilde{\mathbf h}$ is forced by the acoustic and velocity perturbations.
It is this two-way coupling, mediated by $\mathbf n$, that disperses the
energy and prevents the concentration of high frequencies which would
otherwise lead to a finite-time breakdown.

\medskip
\noindent\textbf{The Diophantine--Poincar\'e inequality.}
To convert this structure into quantitative estimates we use the following
Poincar\'e-type inequality with loss of derivatives (Lemma
\ref{Diophantine-inequality}): if $\mathbf n$ satisfies
\eqref{Diophantine0}, then for every $s\geq0$ and every $f$ of zero mean
on $\mathbb T^3$,
\begin{equation}\label{Poincare-Diophantine}
	\|f\|_{H^{s}(\mathbb T^3)}
	\leq C\,\|(\mathbf n\cdot\nabla)f\|_{H^{s+r}(\mathbb T^3)}.
\end{equation}
Inequality \eqref{Poincare-Diophantine} is the tool that turns the hidden
dissipation generated by the wave structure into Sobolev control of $p$
and $\mathbf h$; the loss of $r$ derivatives it entails is the reason for
the gap $N\geq M+2(r+1)$ between the two levels of regularity in
Theorem \ref{th:main1}.

\subsection{Difficulties and strategy of the proof}
Given smooth data, the local well-posedness of \eqref{no-heat-mhd3}
follows from a standard energy argument (see, e.g., \cite{KawashimaS,MM}),
so that everything rests on global a priori estimates. Two independent
difficulties have to be faced: none of $\rho$, $p$, $\mathbf h$ possesses
any dissipation or damping, and $\rho_0$ is allowed to be of arbitrarily
large oscillation. We describe here the scheme designed to overcome them;
all constants and all the terms abbreviated by ``$\cdots$'' below are made
explicit in Section \ref{Theorem-th:main1}.

\medskip
\noindent\textbf{Step 1: a two-tier energy.}
Because $\rho$ exhibits arbitrarily
large variations  and solves a transport equation, its Sobolev norms may grow in
time, and this growth pollutes every estimate in which $\rho$ appears as a
coefficient. This is a major challenge for the problem under consideration. Our answer is a two-tier scheme: we propagate the highest
($N$-th) order energy of $(\mathbf u,p,\mathbf h)$ while simultaneously
proving \emph{decay} of the intermediate ($M$-th) order energy, and we use
the latter to control the former. At the $M$-th level, the natural energy
estimate reads
\begin{equation}\begin{split}\label{25-6-22-1-1}
		\frac{d}{dt}\big\|\big(\sqrt{\rho}\,\Lambda^M\mathbf u,
		(\gamma\bar P)^{-\frac12}\Lambda^M p,\Lambda^M\mathbf h\big)\big\|^2_{0}
		+\mu\|\nabla\Lambda^M\mathbf u\|_0^2
		\lesssim\|\rho\|^2_{3}\|\Lambda^{M-2}\mathbf u_t\|^2_{0}
		+\|\rho\|^2_{M-1}\|\mathbf u_t\|^2_{L^\infty}+\cdots
\end{split}\end{equation}
(see Lemma \ref{priori-1} with $k=M$ for the details). The two displayed terms on the
right-hand side of \eqref{25-6-22-1-1} are precisely those in which the large density enters, and
they cannot be absorbed directly. To overcome this obstacle,  we shall employ  the positive upper
and lower bounds on $\rho$ together with the interpolation inequality
$$\|\Lambda^{M-2}\mathbf u_t\|_{0}
\lesssim\|\sqrt{\rho}\,\mathbf u_t\|^{\frac{1}{M-1}}_{0}
\|\sqrt{\rho}\,\Lambda^{M-1}\mathbf u_t\|_0^{\frac{M-2}{M-1}},$$
which reduces matters to estimating $\|\sqrt\rho\,\mathbf u_t\|_0$ and
$\|\sqrt\rho\,\Lambda^{M-1}\mathbf u_t\|_0$. These are supplied by
\begin{equation}\label{basic-3-1}
	\frac{d}{dt}\big(\|(\mathbf u,p,\mathbf h)\|_0^2+\|\nabla\mathbf u\|_0^2\big)
	+\|\sqrt{\rho}\,\mathbf u_t\|_0^2+\|\nabla\mathbf u\|_0^2
	\lesssim\big(1+\|(\mathbf u,p,\mathbf h)\|_2\big)
	\|(\mathbf u,p,\mathbf h)\|_2^3
\end{equation}
and by
\begin{align}\label{25-6-22-3-1}
	&\frac{d}{dt}\Big(\mu\|\Lambda^{M}\mathbf u\|^2_0
	+(\mu+\lambda)\|\mathrm{div}\,\Lambda^{M-1}\mathbf u\|_0^2
	-2\int\Lambda^{M-1}(p+\mathbf h\cdot\mathbf n)\,
	\mathrm{div}\,\Lambda^{M-1}\mathbf u\,dx\nonumber\\
	&\qquad+2\int(\mathbf n\cdot\nabla)\Lambda^{M-1}\mathbf u\cdot
	\Lambda^{M-1}\mathbf h\,dx\Big)
	+\frac32\|\sqrt{\rho}\,\Lambda^{M-1}\mathbf u_t\|^2_{0}\nonumber\\
	&\quad\lesssim\|\rho\|^2_{3}\|\Lambda^{M-2}\mathbf u_t\|^2_{0}
	+\|\rho\|^2_{M-1}\|\mathbf u_t\|^2_{L^\infty}
	+\|\Lambda^{M+1}\mathbf u\|^2_0+\cdots
\end{align}
(see Lemma \ref{priori-2} with $k=M$ for the details). Since $M>3$, the two problematic terms
are of lower order relative to
$\|\sqrt{\rho}\,\Lambda^{M-1}\mathbf u_t\|^2_0$; combining
\eqref{25-6-22-1-1}, \eqref{basic-3-1} and \eqref{25-6-22-3-1} with Young's inequality,  the
Gagliardo--Nirenberg inequality therefore closes this part of the argument
(see Lemma \ref{priori-25-6-27-1} with $k=M$ for the details).

\medskip
\noindent\textbf{Step 2: hidden dissipation for $p$ and $\mathbf h$.}
Following the linear analysis of Subsection \ref{subsec-linear} --- and
drawing on the incompressible non-resistive theory
\cite{Chen-Zhang-Zhou,Xie-Jiu-Liu} --- we exploit the coupling
between the equations for $\mathbf u$, $p$ and $\mathbf h$ together with
the Diophantine condition \eqref{Diophantine0}. Projecting the momentum
equation onto $\mathbf n$ expresses $(\mathbf n\cdot\nabla)p$ in terms of
$\mathbf u$, $\mathbf u_t$ and quadratic terms, and
\eqref{Poincare-Diophantine} then converts this into the dissipative
estimate
\begin{align}\label{26-4-30-1-1}
	\|(p,\mathbf h)\|^2_{M-1-r}
	&\lesssim\|\sqrt{\rho}\,\Lambda^{M-1}\mathbf u_t\|^2_{0}
	+\|\sqrt{\rho}\,\mathbf u_t\|^2_{0}
	+\|\rho\|^2_{M-1}\|\mathbf u_t\|^2_{L^\infty}\nonumber\\
	&\quad+(1+\|\rho\|^2_M)\|(\mathbf u,\mathbf h)\|^4_M
	+\|\Lambda^{M+1}\mathbf u\|_0^2
\end{align}
(see Lemma \ref{priori-3} for the details). This is the quantitative form of the statement
that $p$ and $\mathbf h$ are dissipated, even though their equations
contain no dissipative term.

\medskip
\noindent\textbf{Step 3: decay of the intermediate energy.} 
Under the a priori assumptions
\begin{align}
	&\tfrac{1}{2}\underline{\rho}\leq\rho\leq2\bar\rho
	\quad\text{on }\mathbb T^3\times[0,T],\label{suppose-1-6-10}\\
	&\sup_{t\in[0,T]}\|\rho\|_{M}\leq4\|\rho_0\|_M,\label{rhoH-M-6-10}\\
	&(1+\|\rho_0\|_M)\sup_{t\in[0,T]}
	\|(\mathbf u,p,\mathbf h)\|_{N}^{\frac{2(r+1)}{N-M+r+1}}
	+M_{\rho}^{\frac{N-M+r+1}{N-M}}\sup_{t\in[0,T]}
	\|(\mathbf u,p,\mathbf h)\|_{N}^{\frac{2(r+1)}{N-M}}\leq\epsilon_0,
	\label{2026-6-10-1}
\end{align}
for some $T>0$ and some small $\epsilon_0>0$,   combining Steps 1 and 2 with the Gagliardo-Nirenberg interpolation inequality yields
the intermediate  algebraic decay
\begin{equation}\label{2026-6-10-2}
	\|(\mathbf u,p,\mathbf h)\|^2_M\lesssim
	\Big(t+\big(M_\rho\|(\mathbf u_0,p_0,\mathbf h_0)\|^2_M
	\big)^{-\frac{r+1}{N-M}}\Big)^{-\frac{N-M}{r+1}}.
\end{equation}
In particular, we further exploit the following  time weighted estimate
\begin{align}\label{2026-6-10-3}
	\sup_{t\in[0,T]}\beta(t)\|(\mathbf u,p,\mathbf h)\|^2_M
	+\int_0^T\beta(t)\Big(\|\mathbf u_t\|^2_{M-1}+\|\mathbf u\|_{M+1}^2
	+\|(p,\mathbf h)\|^2_{M-1-r}\Big)dt
	\lesssim\|(\mathbf u_0,p_0,\mathbf h_0)\|^2_M,
\end{align}
where
$\beta(t):=\big(\frac{r+1}{N-M}
\|(\mathbf u_0,p_0,\mathbf h_0)\|_M^{\frac{2(r+1)}{N-M}}t+1
\big)^{\frac{N-M-1}{r+1}}$ (see Lemma \ref{priori-4} for the details). It should be pointed out that,  the weight function $\beta(t)$,
which grows polynomially in time, is a very crucial device that will pay for the
growth of the density in Step 5.

\medskip
\noindent\textbf{Step 4: closing the estimates for the density.}
Feeding \eqref{2026-6-10-2} and \eqref{2026-6-10-3}, along with the
smallness of the data $(\mathbf u_0,p_0,\mathbf h_0)$, into the transport
equation $\eqref{no-heat-mhd3}_1$, we recover
\begin{equation*}
\frac{2}{3}\underline{\rho}\leq\rho\leq\frac32\bar\rho
\ \ \text{on }\mathbb T^3\times[0,T],
\qquad
\sup_{t\in[0,T]}\|\rho\|^2_{M}\leq2\|\rho_0\|^2_{M},
\end{equation*}
together with the top-order bound
$\sup_{\tau\in[0,t]}\|\rho\|^2_{N}\leq A_7\big(1+t\int_0^t
\|\mathbf u\|_{N+1}^2\,d\tau\big)$ (see Lemma \ref{priori-25-6-27-2} for the details). Note
that the intermediate norm of the density does \emph{not} grow, whereas the
top-order one may grow linearly in time.

\medskip
\noindent\textbf{Step 5: the highest-order estimates.}
Repeating the derivation of Step 1 at the $N$-th level gives
\begin{align}\label{2026-6-10-4}
	&\frac{d}{dt}\widetilde{\mathcal E}
	+\frac{2}{3}\underline\rho\,\|\mathbf u_t\|^2_N
	+\frac{4}{3}\|\mathbf u\|_{N+1}^2\nonumber\\
	&\lesssim\|\rho\|^2_{N-1}\|\mathbf u_t\|^2_{2}
	+\big(\|\rho_0\|_M\|\mathbf u\|_3\|\mathbf u\|_{N}
	+\|\rho\|_N\|\mathbf u\|^2_{3}\big)\|\mathbf u\|_N
	+\|\rho\|^2_{N-1}\|\mathbf u\|^4_3\nonumber\\
	&\quad+\|(\mathbf u,p,\mathbf h)\|^2_{3}\|(\mathbf u,p,\mathbf h)\|^2_{N}
	+\|(\mathbf u,p,\mathbf h)\|_{3}\|(\mathbf u,p,\mathbf h)\|^2_{N}\nonumber\\
	&\quad+\big(\|\rho_0\|^{2(N-1)}_{M}+1\big)
	\big(1+\|(\mathbf u,p,\mathbf h)\|_2\big)\|(\mathbf u,p,\mathbf h)\|_2^3,
\end{align}
where $\widetilde{\mathcal E}$ is equivalent to
$\|(\mathbf u,p,\mathbf h)\|^2_N$, in the sense that
$$\|(\mathbf u,p,\mathbf h)\|^2_N\leq\widetilde{\mathcal E}\lesssim
\big(\|\rho_0\|^{2(N-1)}_{M}+1\big)\|(\mathbf u,p,\mathbf h)\|^2_N$$
(see Lemma \ref{priori-combine-2026-6-5} for the details). Closing \eqref{2026-6-10-4}
requires the right-hand side to be integrable in time, and here lies the
final --- and most delicate --- difficulty: by Step 4 the factor
$\|\rho\|^2_{N-1}$ in the first term grows linearly in $t$. The resolution
is the observation that the weighted estimate \eqref{2026-6-10-3} already
provides a time-integrable bound for $\|\mathbf u_t\|^2_{M-1}$ against the
growing weight $\beta(t)$, and that $\beta(t)$ can be chosen so as to
compensate exactly the linear growth of $\|\rho\|^2_{N-1}$. Combining this
with the decay \eqref{2026-6-10-2} of the intermediate energy closes the
highest-order estimate (see Lemma \ref{priori-2026-5-6-1} for the details), hence also the a
priori assumption \eqref{2026-6-10-1}, provided the initial  data are small.
A continuity argument then converts these into uniform bounds
(see Proposition \ref{th:main2}), and  then Theorem \ref{th:main1} follows.

\subsection{Organization of the paper}
Section \ref{sec2} collects the functional inequalities used throughout.
Section \ref{Theorem-th:main1} is devoted to the \emph{a priori} estimates
for smooth solutions of \eqref{no-heat-mhd3}, following the five steps
outlined above. Theorem \ref{th:main1} is proved in Section
\ref{th:main-11}.

\section{Preliminaries}\label{sec2}
This section provides  several functional inequalities to be used in the proof of our main result.
Let us start with the Poincar\'{e} inequalities in the case of the unit torus $\mathbb T^d$.
\begin{lemma}\label{Poincare-inequality}
(\cite{Danchin-Mucha}, Lemma A.1) Let $w:\mathbb T^d\to\mathbb R$ be a nonnegative, nonzero measurable function, and set $w_0:=\int_{\mathbb T^d}w\,dx$. Then, for any $z\in H^1(\mathbb T^d)$, there holds
$$\|z\|_{L^2}\leq\frac{1}{w_0}\Big|\int_{\mathbb T^d}wz\,dx\Big|+\Big(1+\frac{1}{w_0}\|w_0-w\|_{L^2}\Big)\|\nabla z\|_{L^2}.$$
\end{lemma}

The next lemma states a special Poincar\'{e} inequality involving a vector satisfying the Diophantine condition.
\begin{lemma}\label{Diophantine-inequality}
(\cite{Xie-Jiu-Liu}, Lemma 2.1) Assume that $\mathbf n\in\mathbb R^3$ satisfies the Diophantine condition \eqref{Diophantine0} and that $f\in H^{s+r+1}(\mathbb T^3)$. Then, for any $s\geq0$, there is a positive constant $K_1$ depending only on $\mathbb T^{3},r,c,|\mathbf n|,s$ such that
$$\|\Lambda^sf\|_{0}\leq K_1\|(\mathbf n\cdot\nabla)\Lambda^{s+r}f\|_{0}.$$
If, in addition $\int_{\mathbb T^3} fdx=0$, it holds
$$\|f\|_{0}\leq K_2\|(\mathbf n\cdot\nabla)\Lambda^{r}f\|_{0}$$
with a positive constant $K_2$ depending only on $\mathbb T^{3},r,c,|\mathbf n|.$
\end{lemma}

Finally we present several calculus inequalities.
\begin{lemma}\label{Hk-estimate6-20}
(\cite{Kato-Ponce,Kenig-Ponce-Vega}) Let $s\geq0.$ There is a positive constant $K_3$ depending only on $\mathbb T^{3},s,$ such that
$$\|fg\|_{s}\leq K_3\|f\|_{L^\infty}\|g\|_s+K_3\|g\|_{L^\infty}\|f\|_s.$$
\end{lemma}
\begin{lemma}\label{Hk-estimate}
(\cite{Kato-Ponce,Kenig-Ponce-Vega}) Let $s>0$ and $1<p<\infty.$ There is a positive constant $K_4$ depending only on $\mathbb T^{3},s,p,p_1,p_2,p_3,p_4,$ such that
$$\|[\Lambda^s,f]g\|_{L^p}\leq K_4\|\nabla f\|_{L^{p_1}}\|\Lambda^{s-1}g\|_{L^{p_2}}+K_4\|\Lambda^{s}f\|_{L^{p_3}}\|g\|_{L^{p_4}},$$
where $p_2,p_3\in(1,\infty)$ and  $\frac1p=\frac{1}{p_1}+\frac{1}{p_2}=\frac{1}{p_3}+\frac{1}{p_4}.$
\end{lemma}
\begin{lemma}\label{G-N-inequality}
 (\cite{Xie-Jiu-Liu} Lemma 2.3) Let some constants $s,s_1,s_2,p,p_1,p_2,l$ satisfy $s_2\geq s\geq s_1\geq0$, $1\leq p,p_1,p_2\leq\infty,0\leq l\leq1$, and
$$\frac1p-\frac{s}{d}=l(\frac{1}{p_1}-\frac{s_1}{d})+(1-l)(\frac{1}{p_2}-\frac{s_2}{d}).$$ If $\int_{\mathbb T^d}fdx=0,$ there is a positive constant $K_5$ depending only on $\mathbb T^d,s,s_1,s_2,p,p_1,p_2,l$ such that
$$\|\Lambda^sf\|_{L^p}\leq K_5\|\Lambda^{s_1}f\|^{l}_{L^{p_1}}\|\Lambda^{s_2}f\|^{1-l}_{L^{p_2}}.$$
\end{lemma}
\begin{corollary}\label{jiaquanG-N}
Let $w:\mathbb T^d\to\mathbb R$ be a nonnegative, nonzero measurable function, and let the constants $s,s_1,s_2,p,p_1,p_2,l$ satisfy $s_2\geq s\geq s_1\geq0$, $1\leq p,p_1,p_2\leq\infty$, $0\leq l\leq1$, and
$$\frac1p-\frac{s}{d}=l\Big(\frac{1}{p_1}-\frac{s_1}{d}\Big)+(1-l)\Big(\frac{1}{p_2}-\frac{s_2}{d}\Big).$$ If $w\in L^{\frac{p}{p-1}},$ then there holds that
$$\|f\|_{W^{s,p}}\leq K_6\|\Lambda^{s_1}f\|^{l}_{L^{p_1}}\|\Lambda^{s_2}f\|^{1-l}_{L^{p_2}}+K_6\big|\int_{\mathbb T^d} wfdx\big|$$
with a positive constant $K_6$ depending only on $\mathbb T^d,s,s_1,s_2,p,p_1,p_2,l,\|w\|_{L^{\frac{p}{p-1}}},\|w\|^{-1}_{L^1}.$
\end{corollary}
\begin{proof}
Let $w_0:=\int_{\mathbb T^d} wdx$ and $\bar f:=\int_{\mathbb T^d} fdx$. Then it follows from Lemma \ref{G-N-inequality} and H\"{o}lder's inequality and Poincar\'{e}'s inequality that
\begin{align*}
\big\|w_0f-\int_{\mathbb T^d} wfdx\big\|_{W^{s,p}}&\leq\|w_0f-w_0\bar f\|_{W^{s,p}}+|\mathbb T^d|^{\frac1p}\big|\int_{\mathbb T^d} w(f-\bar f)dx\big|\\
&\leq w_0\|f-\bar f\|_{W^{s,p}}+|\mathbb T^d|^{\frac1p}\|w\|_{L^{\frac{p}{p-1}}}\|f-\bar f\|_{L^{p}}\\
&\leq\big(w_0+|\mathbb T^d|^{\frac1p}\|w\|_{L^{\frac{p}{p-1}}}\big)\|f-\bar f\|_{W^{s,p}}\\
&\leq C_p\big(w_0+|\mathbb T^d|^{\frac1p}\|w\|_{L^{\frac{p}{p-1}}}\big)\|\Lambda^{s}(f-\bar f)\|_{L^{p}}\\
&\leq C_pK_5\big(w_0+|\mathbb T^d|^{\frac1p}\|w\|_{L^{\frac{p}{p-1}}}\big)\|\Lambda^{s_1}f\|^{l}_{L^{p_1}}\|\Lambda^{s_2}f\|^{1-l}_{L^{p_2}}.
\end{align*}
Thus, the desired conclusion follows from $$\|f\|_{W^{s,p}}\leq w_0^{-1}\big\|w_0f-\int_{\mathbb T^d} wfdx\big\|_{W^{s,p}}+w_0^{-1}|\mathbb T^d|^{\frac1p}\big|\int_{\mathbb T^d} wfdx\big|.$$
\end{proof}

\section{\emph{A priori} estimates}\label{Theorem-th:main1}
In this section, we shall derive  \emph{a priori} estimates for  smooth solution to  the system \eqref{no-heat-mhd3}. For this purpose, we assume throughout that $(\rho,\mathbf{u},p,\mathbf{h})$ is a smooth solution of  the system \eqref{no-heat-mhd3} on $\mathbb{T}^3\times[0,T]$ for some $T>0$, with $(\rho,\mathbf{u},p,\mathbf{h})\in C([0,T];H^{N}(\mathbb{T}^3))$. 
Here, we first present the basic energy estimates.
\begin{lemma}\label{priori-basic}
Assume $(\rho,\mathbf u,p,\mathbf h)$ is a smooth solution to   the system \eqref{no-heat-mhd3} on $\mathbb T^3\times[0,T]$ for some $T>0$.  Then there holds,
\begin{equation}\label{basic-1}
\frac{d}{dt}\Big(\|(\sqrt{\rho}\mathbf u,\mathbf h)\|^2_{0}+\frac{2}{\gamma-1}\int pdx\Big)=0.
\end{equation}
Moreover, if the conditions \eqref{junzhi-1} and \eqref{chuzhixiaoyu0}  hold, then we have
\begin{equation}\begin{split}\label{basic-2}
&\|\mathbf u\|_1\leq B_1\big(1+\sup_{\tau\in[0,T]}\|\rho\|_{L^\infty}\big)\|\nabla\mathbf u\|_0,\\ &\|\mathbf h\|_1\leq B_2\|\nabla\mathbf h\|_0,\\
&\Big|\int pdx\Big|\leq B_3\sup_{\tau\in[0,T]}\|\rho\|_{L^\infty}\|(\mathbf u,\mathbf h)\|_{0}\|(\nabla\mathbf u,\nabla\mathbf h)\|_0,
\end{split}\end{equation}
where  the positive constants $B_1,B_2$ and $B_3$ depend only on $\mathbb T^{3},\gamma,M_0.$
\end{lemma}
\begin{proof}
Taking the $L^2$-scalar product of the second equation of the system \eqref{no-heat-mhd3} with $\mathbf u$  and
the fourth equation with $\mathbf h$ respectively, combining together, multiplying $\eqref{no-heat-mhd3}_{3}$ by $\frac{1}{\gamma-1}$ and  integrating on $\mathbb T^3$, integrating by parts and then  employing $\eqref{no-heat-mhd3}_{1,5}$ and the following facts:
$$\int\big((\mathbf h\cdot\nabla)\mathbf h\cdot \mathbf u+(\mathbf h\cdot\nabla)\mathbf u\cdot \mathbf h\big)dx=0,$$
$$\int\Big(\frac12(\mathbf u\cdot\nabla)(|\mathbf h|^2)+(\mathbf u\cdot\nabla)\mathbf h\cdot \mathbf h+\mathrm{div}\mathbf u\,|\mathbf h|^2\Big)dx=0,$$
$$\int\big(\mu|\nabla\mathbf u|^2+(\mu+\lambda)(\mathrm{div}\mathbf u)^2\big)dx=\int\mathcal Q(\mathbf u)dx,$$
we  conclude that \eqref{basic-1} holds.

Integrating \eqref{basic-1} over $(0,t)$ and then  using the condition \eqref{chuzhixiaoyu0}, yield
\begin{align}\label{p-poincare}\Big|\int pdx\Big|=\frac{\gamma-1}{2}\|(\sqrt{\rho}\mathbf u,\mathbf h)\|^2_{0}\leq\frac{\gamma-1}{2}\sup_{\tau\in[0,T]}\|\rho\|_{L^\infty}\|(\mathbf u,\mathbf h)\|^2_{0}.\end{align}
Obviously, it follows from the equations $\eqref{no-heat-mhd3}_{1,2,4}$ and the condition \eqref{junzhi-1} that
\begin{equation}\label{p-poincare-AAA}\int\rho dx =M_0,\ \int\rho \mathbf udx=\int \mathbf hdx=\mathbf{0}.\end{equation}
This along  with Lemma \ref{Poincare-inequality} gives rise to
 $$\|\mathbf u\|_1\leq \Big(2+\frac{1}{M_0}\|M_0-\rho\|_{0}\Big)\|\nabla \mathbf u\|_{0},$$
which  together with $\|\rho\|_0\leq|\mathbb T^3|^\frac12\|\rho\|_{L^\infty},$ yields the first inequality  in \eqref{basic-2}. Obviously,
the second inequality in \eqref{basic-2} follows from Poincar\'{e}'s inequality. Then using the first two inequalities in \eqref{basic-2} and  \eqref{p-poincare}, we finally obtain the last inequality in \eqref{basic-2}. This completes the proof of Lemma \ref{priori-basic}.
\end{proof}
\begin{lemma}\label{priori-basic-1}
 Assume $(\rho,\mathbf u,p,\mathbf h)$ is a smooth solution to   the system \eqref{no-heat-mhd3} on $\mathbb T^3\times[0,T]$ for some $T>0$, and
\begin{equation}\label{suppose-1}
\frac{1}{2}\underline{\rho}\leq\rho\leq2\bar\rho,\ \ \text{in}\ \mathbb T^3\times[0,T].
\end{equation}
Then there hold,
\begin{equation}\label{basic-3}
\frac{d}{dt}\mathcal A (t)+\|\sqrt\rho\mathbf u_t\|_0^2+\|\nabla\mathbf u\|_0^2\lesssim\big(1+\|(\mathbf u,p,\mathbf h)\|_2\big)\|(\mathbf u,p,\mathbf h)\|_2^3,
\end{equation}
and
\begin{equation}\label{basic-3-4-9}
A_1\big(\|(\mathbf u,p,\mathbf h)\|_0^2+\|\nabla\mathbf u\|_0^2\big)\leq\mathcal A(t)\leq A_2\big(\|(\mathbf u,p,\mathbf h)\|_0^2+\|\nabla\mathbf u\|_0^2\big),
\end{equation}
where the two positive constants $A_1,A_2$ depend only on $\mu,\lambda,\gamma,\bar P,|\mathbf n|,\underline{\rho},\bar\rho,$ and
\begin{align*}\mathcal A(t):=&\mu_3\|(\sqrt{\rho}\mathbf u(\cdot,t),(\gamma\bar P)^{-\frac12} p(\cdot,t),\mathbf h(\cdot,t))\|^2_{0}+\mu\|\nabla\mathbf u(\cdot,t)\|_0^2+(\mu+\lambda)\|\mathrm{div}\mathbf u(\cdot,t)\|_0^2\\
&-2\int \big(p(x,t)+\mathbf h(x,t)\cdot\mathbf n\big)\mathrm{div}\mathbf u(x,t)dx+2\int\mathbf h(x,t)\cdot (\mathbf n\cdot\nabla)\mathbf u(x,t)dx
\end{align*}
with $\mu_3:=\max\big\{(4|\mathbf n|^2+\gamma\bar P+1)\mu^{-1},\frac{4\gamma\bar P}{\mu}+1,\frac{16|\mathbf n|^2}{\mu}+1\big\}.$
\end{lemma}
\begin{proof}
Performing  the $L^2$-scalar product of the second equation of the system \eqref{no-heat-mhd3}  with $\mathbf u$,  the third equation of the system \eqref{no-heat-mhd3}  with $\frac{p}{\gamma\bar P}$,  and
the fourth equation with $\mathbf h$ respectively, combining together, integrating by parts and using $\eqref{no-heat-mhd3}_{1,5}$, H\"{o}lder's inequality and Sobolev's embedding inequality, we infer that
\begin{align}\label{basic-4}
&\frac12\frac{d}{dt}\|(\sqrt{\rho}\mathbf u,(\gamma\bar P)^{-\frac12} p,\mathbf h)\|^2_{0}+\mu\|\nabla\mathbf u\|_0^2+(\mu+\lambda)\|\mathrm{div}\mathbf u\|_0^2\nonumber\\
&=(\gamma\bar P)^{-1}\int\left(-(\mathbf u\cdot\nabla)p-\gamma p\mathrm{div}\mathbf u+(\gamma-1)\mathcal Q(\mathbf u)\right)pdx\nonumber\\
&\lesssim\|\mathbf u\|_{0}\|\nabla p\|_0\|p\|_{L^\infty}+\|\nabla\mathbf u\|_{0}\|p\|_0\|p\|_{L^\infty}+\|\nabla\mathbf u\|_0^2\|p\|_{L^\infty}\nonumber\\&\lesssim\|(\mathbf u,p)\|_2^3.
\end{align}
Taking the $L^2$-scalar product of the second equation of the system \eqref{no-heat-mhd3} with $\mathbf u_t$, integrating by parts and using H\"{o}lder's inequality, yield
\begin{align}\label{basic-5}
&\frac12\frac{d}{dt}\big(\mu\|\nabla\mathbf u\|_0^2+(\mu+\lambda)\|\mathrm{div}\mathbf u\|_0^2\big)+\|\sqrt\rho\mathbf u_t\|_0^2\nonumber\\
&=-\int\rho(\mathbf u\cdot\nabla)\mathbf u\cdot\mathbf u_tdx-\frac12\int(\mathbf u_t\cdot\nabla)(|\mathbf h|^2)dx+\int (p+\mathbf h\cdot\mathbf n)\mathrm{div}\mathbf u_tdx\nonumber\\
&\quad-\int\mathbf h\cdot (\mathbf n\cdot\nabla)\mathbf u_tdx+\int (\mathbf h\cdot\nabla)\mathbf h\cdot\mathbf u_tdx\nonumber\\
&\leq\big(\|\rho^{\frac12}(\mathbf u\cdot\nabla)\mathbf u\|_0+\|\rho^{-\frac12}\nabla|\mathbf h|^2\|_0+\|\rho^{-\frac12}(\mathbf h\cdot\nabla)\mathbf h\|_0\big)\|\sqrt\rho\mathbf u_t\|_0\nonumber\\
&\quad+\frac{d}{dt}\Big(\int (p+\mathbf h\cdot\mathbf n)\mathrm{div}\mathbf udx\Big)-\frac{d}{dt}\Big(\int\mathbf h\cdot (\mathbf n\cdot\nabla)\mathbf udx\Big)\nonumber\\
&\quad-\int (p_t+\mathbf h_t\cdot\mathbf n)\mathrm{div}\mathbf udx+\int\mathbf h_t\cdot (\mathbf n\cdot\nabla)\mathbf udx.
\end{align}
For the last two terms on the right-hand side of the  inequality \eqref{basic-5}, using equations $\eqref{no-heat-mhd3}_{3,4}$ and H\"{o}lder's inequality  and Sobolev's embedding inequality, we obtain
\begin{align*}
-\int (p_t+\mathbf h_t\cdot\mathbf n)\mathrm{div}\mathbf udx&=\int\Big((\mathbf u\cdot\nabla)p+\gamma\bar P\mathrm{div}\mathbf u+\gamma p\mathrm{div}\mathbf u-(\gamma-1)\mathcal Q(\mathbf u)+(\mathbf u\cdot \nabla)\mathbf h\cdot\mathbf n\\
&\quad-(\mathbf h\cdot \nabla)\mathbf u\cdot\mathbf n-(\mathbf n\cdot \nabla)\mathbf u\cdot\mathbf n+\mathbf h\cdot\mathbf n\mathrm{div} \mathbf u+|\mathbf n|^2\mathrm{div} \mathbf u\Big)\mathrm{div}\mathbf udx\\
&\leq C\|(\mathbf u,p,\mathbf h,\nabla\mathbf u)\|_{L^6}\|(\nabla\mathbf u, \nabla p,\nabla\mathbf h)\|_{L^3}\|\nabla\mathbf u\|_0+(2|\mathbf n|^2+\gamma\bar P)\|\nabla\mathbf u\|^2_0\\
&\leq C\|(\mathbf u,p,\mathbf h)\|^3_{2}+(2|\mathbf n|^2+\gamma\bar P)\|\nabla\mathbf u\|^2_0,\\
\int\mathbf h_t\cdot (\mathbf n\cdot\nabla)\mathbf udx&=\int\Big(-(\mathbf u\cdot \nabla)\mathbf h
+(\mathbf h\cdot \nabla)\mathbf u+(\mathbf n\cdot \nabla)\mathbf u-\mathbf h\mathrm{div} \mathbf u-\mathbf n\mathrm{div} \mathbf u\Big)\cdot (\mathbf n\cdot\nabla)\mathbf udx\\
&\leq C\|(\mathbf u,\mathbf h)\|_2^3+2|\mathbf n|^2\|\nabla\mathbf u\|^2_0.
\end{align*}
Then inserting the above two estimates into \eqref{basic-5}, employing the condition \eqref{suppose-1},  H\"{o}lder's and Young's inequalities,  and Sobolev's embedding  inequality, we have
\begin{align*}
\frac12&\frac{d}{dt}\Big(\mu\|\nabla\mathbf u\|_0^2+(\mu+\lambda)\|\mathrm{div}\mathbf u\|_0^2-2\int (p+\mathbf h\cdot\mathbf n)\mathrm{div}\mathbf udx+2\int\mathbf h\cdot (\mathbf n\cdot\nabla)\mathbf udx\Big)+\|\sqrt\rho\mathbf u_t\|_0^2\\
&\leq \big(2\bar\rho^\frac12\|\mathbf u\|_{L^\infty}\|\nabla\mathbf u\|_0+4\underline{\rho}^{-\frac12}\|\mathbf h\|_{L^\infty}\|\nabla\mathbf h\|_0\big)\|\sqrt\rho\mathbf u_t\|_0+C\|\mathbf u,p,\mathbf h\|_2^3+(4|\mathbf n|^2+\gamma\bar P)\|\nabla\mathbf u\|^2_0\\
&\leq C(1+\|(\mathbf u,p,\mathbf h)\|_2)\|(\mathbf u,p,\mathbf h)\|_2^3+(4|\mathbf n|^2+\gamma\bar P)\|\nabla\mathbf u\|^2_0+\frac12\|\sqrt\rho\mathbf u_t\|^2_0.
\end{align*}
Denoting  $$\mu_3:=\max\{(4|\mathbf n|^2+\gamma\bar P+1)\mu^{-1},\frac{4\gamma\bar P}{\mu}+1,\frac{16|\mathbf n|^2}{\mu}+1\},$$
multiplying inequality \eqref{basic-4} by $\mu_3$ and then adding it to the above inequality, we get
\begin{align*}
\frac12&\frac{d}{dt}\mathcal A(t)+\frac12\|\sqrt\rho\mathbf u_t\|_0^2+(\mu_3\mu-4|\mathbf n|^2-\gamma\bar P)\|\nabla\mathbf u\|_0^2+(\mu+\lambda)\|\mathrm{div}\mathbf u\|_0^2\\
&\lesssim\big(1+\|(\mathbf u,p,\mathbf h)\|_2\big)\|(\mathbf u,p,\mathbf h)\|_2^3,
\end{align*}
which implies  that \eqref{basic-3}.

Thanks to
\begin{align*}
&\big|-2\int (p+\mathbf h\cdot\mathbf n)\mathrm{div}\mathbf udx+2\int\mathbf h\cdot (\mathbf n\cdot\nabla)\mathbf udx\big|\\&\leq\big(2\|p\|_0+4|\mathbf n|\|\mathbf h\|_0\big)\|\nabla\mathbf u\|_0\\
&\leq\frac{4\gamma\bar P}{\mu}\|(\gamma\bar P)^{-\frac12} p\|_0^2+\frac{16|\mathbf n|^2}{\mu}\|\mathbf h\|^2_0+\frac{\mu}{2}\|\nabla\mathbf u\|^2_0,
\end{align*}
it then  follows form \eqref{suppose-1} and $\mu+\lambda>0$, that
\begin{align*}\mathcal A&\geq\mu_3\|\sqrt{\rho}\mathbf u\|_0^2+\Big(\mu_3-\frac{4\gamma\bar P}{\mu}\Big)\|(\gamma\bar P)^{-\frac12} p\|_0^2+\Big(\mu_3-\frac{16|\mathbf n|^2}{\mu}\Big)\|\mathbf h\|^2_0+\frac\mu2\|\nabla\mathbf u\|_0^2\\
&\geq\min\big\{\frac{\underline{\rho}}{2},1\big\}\|(\mathbf u,p,\mathbf h)\|_0^2+\frac\mu2\|\nabla\mathbf u\|_0^2,\\
\mathcal A&\leq\mu_3\Big(1+\bar\rho+\frac1{\gamma\bar P}+\frac4\mu+\frac{16|\mathbf n|^2}{\mu}\Big)\|(\mathbf u,p,\mathbf h)\|_0^2+(3\mu+\lambda)\|\nabla\mathbf u\|_0^2.
\end{align*}
Consequently, we complete the proof of  Lemma \ref{priori-basic-1}.
\end{proof}
In what follows, we shall show the higher-order energy estimates.
\begin{lemma}\label{priori-1}
Let $k\geq3.$ Assume $(\rho,\mathbf u,p,\mathbf h)$ is a smooth solution to   the system \eqref{no-heat-mhd3} on $\mathbb T^3\times(0,T]$ for some $T>0$,  then there holds,
\begin{align}\label{25-6-22-1}
\frac{d}{dt}&\|(\sqrt{\rho}\Lambda^k\mathbf u,(\gamma\bar P)^{-\frac12}\Lambda^k p,\Lambda^k\mathbf h)\|^2_{0}+\mu\|\nabla\Lambda^k\mathbf u\|_0^2\nonumber\\
&\lesssim\|\rho\|^2_{3}\|\Lambda^{k-2}\mathbf u_t\|^2_{0}+\|\rho\|^2_{k-1}\|\mathbf u_t\|^2_{L^\infty}+\big(\|\rho\|_3\|\mathbf u\|_3\|\mathbf u\|_{k}+\|\rho\|_k\|\mathbf u\|^2_{3}\big)\|\mathbf u\|_k\nonumber\\
&\quad+\|(\mathbf u,p,\mathbf h)\|^2_{3}\|(\mathbf u,p,\mathbf h)\|^2_{k}+\|(\mathbf u,p,\mathbf h)\|_{3}\|(\mathbf u,p,\mathbf h)\|^2_{k}.
\end{align}
\end{lemma}
\begin{proof}
We apply the operator $\Lambda^k$ to $\eqref{no-heat-mhd3}_{2,3,4}$, take
the $L^2$-scalar products of the resulting three equations with
$\Lambda^k\mathbf u$, $(\gamma\bar P)^{-1}\Lambda^kp$ and
$\Lambda^k\mathbf h$ respectively, and add them. Integrating by parts and
using $\eqref{no-heat-mhd3}_{1,5}$, we deduce that
\begin{align}\label{6-16-1}
\frac12&\frac{d}{dt}\|(\sqrt{\rho}\Lambda^k\mathbf u,(\gamma\bar P)^{-\frac12}\Lambda^k p,\Lambda^k\mathbf h)\|^2_{0}
+\mu\|\nabla\Lambda^k\mathbf u\|_0^2+(\mu+\lambda)\|\Lambda^k\mathrm{div}\mathbf u\|_0^2\nonumber\\
&=-\int[\Lambda^k,\rho]\mathbf u_t\cdot\Lambda^k\mathbf udx\nonumber\\
&\quad-\int[\Lambda^k,\rho \mathbf u\cdot\nabla]\mathbf u\cdot\Lambda^k\mathbf udx+\frac12\int\Lambda^k(|\mathbf h|^2)\Lambda^k\mathrm{div}\mathbf udx+\int[\Lambda^k,\mathbf h\cdot\nabla]\mathbf h\cdot\Lambda^k \mathbf udx\nonumber\\
&\quad -\frac1{\gamma\bar P}\int\Lambda^k((\mathbf u\cdot\nabla)p)\Lambda^kpdx-\frac1{\bar P}\int \Lambda^k(p\mathrm{div}\mathbf u)\Lambda^kpdx+\frac{\gamma-1}{\gamma\bar P}\int\Lambda^k\mathcal Q(\mathbf u)\Lambda^k pdx\nonumber\\
&\quad -\int\Lambda^k((\mathbf u\cdot\nabla)\mathbf h)\cdot\Lambda^k\mathbf hdx+\int\Lambda^k((\mathbf h\cdot\nabla)\mathbf u)\cdot\Lambda^k\mathbf hdx-\int\Lambda^k(\mathbf h\mathrm{div}\mathbf u)\cdot\Lambda^k\mathbf hdx
\nonumber\\&\triangleq\sum_{i=1}^{10}\textrm{I}_{i},
\end{align}
where we have used $\mathrm{div}\mathbf h=0$ together with the following three cancellations:
\begin{align*}
&\int\big((\mathbf h\cdot\nabla)\Lambda^k\mathbf h\cdot\Lambda^k\mathbf u
+(\mathbf h\cdot\nabla)\Lambda^k\mathbf u\cdot\Lambda^k\mathbf h\big)dx=0,\\
&\int\Lambda^{k}(\mathbf h\cdot\mathbf n)\Lambda^{k}\mathrm{div}\mathbf udx
-\int\big(\mathbf n\cdot\Lambda^{k}\mathbf h\big)\Lambda^{k}\mathrm{div}\mathbf udx=0,\\
&\int\Lambda^k\big((\mathbf n\cdot\nabla)\mathbf h\big)\cdot\Lambda^k\mathbf udx
+\int(\mathbf n\cdot\nabla)\Lambda^k\mathbf u\cdot\Lambda^k\mathbf hdx=0,
\end{align*}
the last of which expresses the skew-symmetry of $\mathbf n\cdot\nabla$ on $\mathbb T^3$. It is precisely these two $\mathbf n$-cancellations that remove the only terms of \eqref{no-heat-mhd3} that are linear in the perturbation.

We bound term by term above in what follows. For $\textrm{I}_{1}$, using integrating by parts,  H\"{o}lder's and  Young's inequalities and Sobolev's embedding inequality, Lemma \ref{Hk-estimate} and Poincar\'{e}'s inequality, we have
\begin{align*}
|\textrm{I}_{1}|=&\big|-\int[\Lambda^k,\rho_0]\mathbf u_t\cdot\Lambda^k\mathbf udx\big|\nonumber\\
&=\big|-\int[\Lambda^{k-1},\rho_0]\mathbf u_t\cdot\Lambda^{k+1}\mathbf u-\Lambda^{k-2}\mathbf u_t\cdot[\Lambda^{2},\rho_0]\Lambda^{k}\mathbf u+\Lambda^{k-2}\mathbf u_t\cdot[\Lambda,\rho_0]\Lambda^{k+1}\mathbf udx\big|\nonumber\\
&\leq\|[\Lambda^{k-1},\rho_0]\mathbf u_t\|_{0}\|\Lambda^{k+1}\mathbf u\|_{0}+\|\Lambda^{k-2}\mathbf u_t\|_0\|[\Lambda^{2},\rho_0]\Lambda^{k}\mathbf u\|_0+\|\Lambda^{k-2}\mathbf u_t\|_0\|[\Lambda,\rho_0]\Lambda^{k+1}\mathbf u\|_0\nonumber\\
&\lesssim(\|\nabla\rho_0\|_{L^\infty}\|\Lambda^{k-2}\mathbf u_t\|_{0}+\|\rho_0\|_{k-1}\|\mathbf u_t\|_{L^\infty})\|\Lambda^{k+1}\mathbf u\|_{0}+(\|\nabla\rho_0\|_{L^\infty}\|\Lambda^{k}\mathbf u\|_{1}\nonumber\\
&\quad+\|\rho_0\|_{W^{2,3}}\|\Lambda^{k}\mathbf u\|_{L^6})\|\Lambda^{k-2}\mathbf u_t\|_{0}+\|\rho_0\|_3\|\Lambda^{k+1}\mathbf u\|_{0}\|\Lambda^{k-2}\mathbf u_t\|_{0}\nonumber\\
&\lesssim(\|\rho_0\|_{3}\|\Lambda^{k-2}\mathbf u_t\|_{0}+\|\rho_0\|_{k-1}\|\mathbf u_t\|_{L^\infty})\|\Lambda^{k+1}\mathbf u\|_{0}\nonumber\\
&\leq C_{\delta}(\|\rho_0\|^2_{3}\|\Lambda^{k-2}\mathbf u_t\|^2_{0}+\|\rho_0\|^2_{k-1}\|\mathbf u_t\|^2_{L^\infty})+\delta\|\Lambda^{k}\nabla\mathbf u\|^2_{0}.
\end{align*}
Note here, for the estimate of $\|[\Lambda,\rho_0]\Lambda^{k+1}\mathbf u\|_0$ in the above inequality, we have used the following classical estimate:
\begin{equation*}
\|[\Lambda,g]f\|_0\leq C\|g\|_3\|f\|_0.
\end{equation*}
For $\textrm{I}_{2}$, it follows from  H\"{o}lder's inequality and Sobolev's embedding inequality, Lemmas \ref{Hk-estimate6-20} and \ref{Hk-estimate}, that
\begin{align*}
\textrm{I}_{2}&\leq\|[\Lambda^k,\rho \mathbf u\cdot\nabla]\mathbf u\|_0\|\Lambda^k\mathbf u\|_0
\\&\leq(\|\nabla(\rho \mathbf u)\|_{L^\infty}\|\mathbf u\|_{k}+\|\rho\mathbf u\|_{k}\|\nabla\mathbf u\|_{L^\infty})\|\Lambda^k\mathbf u\|_0\\
&\lesssim\big(\|\nabla\rho\|_{L^\infty}\|\mathbf u\|_{L^\infty}+\|\rho\|_{L^\infty}\|\nabla\mathbf u\|_{L^\infty}\big)\|\mathbf u\|_{k}+\big(\|\rho\|_{L^\infty}\|\mathbf u\|_k+\|\rho\|_{k}\|\mathbf u\|_{L^\infty}\big)\|\nabla\mathbf u\|_{L^\infty}\|\mathbf u\|_k\\
&\lesssim\big(\|\rho\|_3\|\mathbf u\|_3\|\mathbf u\|_{k}+\|\rho\|_k\|\mathbf u\|^2_{3}\big)\|\mathbf u\|_k.
\end{align*}
Similarly, for $\textrm{I}_{4}$, it holds
$$|\textrm{I}_{4}|\lesssim\|\nabla\mathbf h\|_{L^\infty}\|\mathbf h\|_k\|\Lambda^k\mathbf u\|_0\lesssim\|(\mathbf u,\mathbf h)\|_3\|(\mathbf u,\mathbf h)\|^2_k.$$
For $\textrm{I}_{3},\textrm{I}_{6},\textrm{I}_{7},\textrm{I}_{9}$ and $\textrm{I}_{10}$, using  H\"{o}lder's and  Young's inequalities, and Sobolev's embedding inequality, Lemma \ref{Hk-estimate6-20} and Poincar\'{e}'s inequality, we conclude that
\begin{align*}
\textrm{I}_{3}&\lesssim\|\mathbf h\|_{L^\infty}\|\mathbf h\|_{k}\|\nabla\Lambda^k\mathbf u\|_0\leq C_\delta\|\mathbf h\|^2_{3}\|\mathbf h\|^2_{k}+\delta\|\nabla\Lambda^k\mathbf u\|^2_0,\\
\textrm{I}_{6}&\lesssim\big(\|p\|_{L^\infty}\|\nabla\mathbf u\|_k+\|p\|_{k}\|\nabla\mathbf u\|_{L^\infty}\big)\|p\|_k\\&\lesssim\|p\|_2\|\nabla\Lambda^k\mathbf u\|_0\|p\|_k+\|\mathbf u\|_3\|p\|_k^2\\
&\leq C_\delta\|p\|^2_2\|p\|^2_k+C\|\mathbf u\|_3\|p\|_k^2+\delta\|\nabla\Lambda^k\mathbf u\|^2_0,\\
\textrm{I}_{7}&\lesssim\|\nabla\mathbf u\|_{L^\infty}\|\nabla\mathbf u\|_k\|p\|_k\leq C_\delta\|\mathbf u\|^2_3\|p\|^2_k+\delta\|\nabla\Lambda^k\mathbf u\|^2_0,\\
\textrm{I}_{9}&\lesssim\|\mathbf h\|_2\|\nabla\Lambda^k\mathbf u\|_0\|\mathbf h\|_k+\|\mathbf u\|_3\|\mathbf h\|_k^2\\&\leq C_\delta\|\mathbf h\|^2_2\|\mathbf h\|^2_k+C\|\mathbf u\|_3\|\mathbf h\|_k^2+\delta\|\nabla\Lambda^k\mathbf u\|^2_0,\\
\textrm{I}_{10}&\leq C_\delta\|\mathbf h\|^2_2\|\mathbf h\|^2_k+C\|\mathbf u\|_3\|\mathbf h\|_k^2+\delta\|\nabla\Lambda^k\mathbf u\|^2_0.
\end{align*}
For $\textrm{I}_{5}$ and  $\textrm{I}_{8},$ we deduces by integrating by parts,  H\"{o}lder's  inequality and Sobolev's embedding inequality, Lemmas \ref{Hk-estimate6-20} and  \ref{Hk-estimate}, that
\begin{equation*}\begin{split}
\textrm{I}_{8}&=-\int(\mathbf u\cdot\nabla)\Lambda^k\mathbf h\cdot\Lambda^k\mathbf hdx-\int[\Lambda^k,\mathbf u\cdot\nabla]\mathbf h\cdot\Lambda^k\mathbf hdx\\
&\leq\frac12\big|\int\mathrm{div}\mathbf u|\Lambda^k\mathbf h|^2dx\big|+\|[\Lambda^k,\mathbf u\cdot\nabla]\mathbf h\|_0\|\Lambda^k\mathbf h\|_0\\
&\lesssim\|\nabla\mathbf u\|_{L^\infty}\|\mathbf h\|^2_k+\|\nabla\mathbf h\|_{L^\infty}\|\mathbf u\|_k\|\mathbf h\|^2_k
\\&\lesssim\|(\mathbf u, \mathbf h)\|_3\|(\mathbf u, \mathbf h)\|_k^2,
\end{split}\end{equation*}
and
\begin{equation*}
|\textrm{I}_{5}|\lesssim\|(\mathbf u,  p)\|_3\|(\mathbf u, p)\|_k^2.
\end{equation*}
Then plugging the above estimates into \eqref{6-16-1} and  choosing $\delta$ small enough, we obtain the desired conclusion. Thus, this completes the proof of  Lemma \ref{priori-1}.\end{proof}
\begin{lemma}\label{priori-2}
Let $k\geq3.$ Under the assumptions in Lemma \ref{priori-basic-1},  then there holds
\begin{align}\label{25-6-22-3}
&\frac{d}{dt}\Big(\mu\|\Lambda^{k}\mathbf u\|^2_0
+(\mu+\lambda)\|\mathrm{div}\Lambda^{k-1}\mathbf u\|_0^2
-2\int\Lambda^{k-1}(p+\mathbf h\cdot\mathbf n)\mathrm{div}\Lambda^{k-1}\mathbf udx\nonumber\\
&\qquad+2\int(\mathbf n\cdot\nabla)\Lambda^{k-1}\mathbf u\cdot\Lambda^{k-1}\mathbf hdx\Big)\nonumber\\
&\quad+\frac32\|\sqrt{\rho}\Lambda^{k-1}\mathbf u_t\|^2_{0}\nonumber\\
&\leq C\Big(\|\rho\|^2_{3}\|\Lambda^{k-2}\mathbf u_t\|^2_{0}
+\|\rho\|^2_{k-1}\|\mathbf u_t\|^2_{L^\infty}
+\|(\mathbf u,\mathbf h)\|^2_{3}\|(\mathbf u,\mathbf h)\|^2_{k}\nonumber\\
&\qquad\quad+\|\rho\|^2_{k-1}\|\mathbf u\|^4_3
+\|(\mathbf u,p,\mathbf h)\|_{3}\|(\mathbf u,p,\mathbf h)\|^2_{k}\Big)\nonumber\\
&\quad+A_3\|\Lambda^{k+1}\mathbf u\|^2_0,
\end{align}
where the positive constant $A_3$ depends only on $\mathbb T^3,\gamma,\bar P,|\mathbf n|.$
\end{lemma}
\begin{proof} Applying the operator $\Lambda^{k-1}$ to $\eqref{no-heat-mhd3}_{2}$, taking $L^2$-scalar product
with $\Lambda^{k-1}\mathbf u_t,$ and then  integrating by parts, we deduce that
\begin{align}\label{6-19-2}
&\frac12\frac{d}{dt}\big(\mu\|\Lambda^{k}\mathbf u\|^2_0+(\mu+\lambda)\|\mathrm{div}\Lambda^{k-1}\mathbf u\|_0^2\big)+\|\sqrt{\rho}\Lambda^{k-1}\mathbf u_t\|^2_{0}\nonumber\\
&=-\int[\Lambda^{k-1},\rho]\mathbf  u_t\cdot\Lambda^{k-1}\mathbf u_tdx-\int\Lambda^{k-1}\big((\rho\mathbf u\cdot\nabla)\mathbf u\big)\cdot\Lambda^{k-1}\mathbf u_tdx-\frac12\int(\Lambda^{k-1}\mathbf u_t\cdot\nabla)\Lambda^{k-1}(|\mathbf h|^2)dx\nonumber\\
&\quad+\int\Lambda^{k-1}(p+\mathbf h\cdot\mathbf n)\mathrm{div}\Lambda^{k-1}\mathbf u_tdx-\int\Lambda^{k-1}\mathbf h\cdot(\mathbf n\cdot\nabla)\Lambda^{k-1}\mathbf u_tdx+\int\Lambda^{k-1}((\mathbf h\cdot\nabla)\mathbf h)\cdot\Lambda^{k-1}\mathbf u_tdx\nonumber\\
&\triangleq\sum_{i=1}^6\textrm{II}_{i}.
\end{align}
In what follows, we bound term by term from the above inequality. For $\textrm{II}_{1},$  similar to the derivation of  $\textrm{I}_{1}$,   
we arrive at
\begin{equation*}\begin{split}
|\textrm{II}_{1}|&\leq\|[\Lambda^{k-1},\rho]\mathbf u_t\|_0\|\Lambda^{k-1}\mathbf u_t\|_0\\&\leq2\underline{\rho}^{-1}\|[\Lambda^{k-1},\rho]\mathbf u_t\|_0\|\sqrt\rho\Lambda^{k-1}\mathbf u_t\|_0\\
&\leq C_\delta\|\rho\|^2_{3}\|\Lambda^{k-2}\mathbf u_t\|^2_{0}+C_\delta\|\rho\|^2_{k-1}\|\mathbf u_t\|^2_{L^\infty}+\delta\|\sqrt\rho\Lambda^{k-1}\mathbf u_t\|^2_{0}.
\end{split}\end{equation*}
For $\textrm{II}_{2},\textrm{II}_{3}$ and $\textrm{II}_{6},$ it follows from H\"{o}lder's and Young's inequalities, Sobolev's embedding inequality and Lemma \ref{Hk-estimate6-20}, that
\begin{equation*}\begin{split}
|\textrm{II}_{2}|&\leq2\underline{\rho}^{-1}\|\rho \mathbf u\cdot\nabla\mathbf u\|_{k-1}\|\sqrt\rho\Lambda^{k-1}\mathbf u_t\|_0\\
&\lesssim\Big(\|\rho\|_{L^\infty}\|\mathbf u\|_{L^\infty}\|\nabla\mathbf u\|_{k-1}+\big(\|\rho\|_{L^\infty}\|\mathbf u\|_{k-1}+\|\rho\|_{k-1}\|\mathbf u\|_{L^\infty}\big)\|\nabla\mathbf u\|_{L^\infty}\Big)\|\sqrt\rho\Lambda^{k-1}\mathbf u_t\|_0\\
&\lesssim\big(\|\mathbf u\|_{3}\|\mathbf u\|_{k}+\|\rho\|_{k-1}\|\mathbf u\|^2_3\big)\|\sqrt\rho\Lambda^{k-1}\mathbf u_t\|_0\\
&\leq C_\delta\|\mathbf u\|^2_{3}\|\mathbf u\|^2_{k}+C_\delta\|\rho\|^2_{k-1}\|\mathbf u\|^4_3+\delta\|\sqrt\rho\Lambda^{k-1}\mathbf u_t\|^2_{0},\\
|\textrm{II}_{3}|&\leq C_\delta\|\mathbf h\|^2_{3}\|\mathbf h\|^2_{k}+\delta\|\sqrt\rho\Lambda^{k-1}\mathbf u_t\|^2_{0},\\
|\textrm{II}_{6}|&\leq C_\delta\|\mathbf h\|^2_{3}\|\mathbf h\|^2_{k}+\delta\|\sqrt\rho\Lambda^{k-1}\mathbf u_t\|^2_{0}.
\end{split}\end{equation*}
For $\textrm{II}_{4}$ and $\textrm{II}_{5}$,  we conclude  by integrating by parts and using $\eqref{no-heat-mhd3}_4$, H\"{o}lder's inequality and Sobolev's embedding inequality and Lemma \ref{Hk-estimate6-20},  that
\begin{align*}
\textrm{II}_{4}&=\frac{d}{dt}\int\Lambda^{k-1}(p+\mathbf h\cdot\mathbf n)\mathrm{div}\Lambda^{k-1}\mathbf udx-\int\Lambda^{k-1}(p_t+\mathbf h_t\cdot\mathbf n)\mathrm{div}\Lambda^{k-1}\mathbf udx\\
&=\frac{d}{dt}\int\Lambda^{k-1}(p+\mathbf h\cdot\mathbf n)\mathrm{div}\Lambda^{k-1}\mathbf udx-\int\Lambda^{k-1}\big(-(\mathbf u\cdot\nabla)p-\gamma\bar P\mathrm{div}\mathbf u-\gamma p\mathrm{div}\mathbf u+(\gamma-1)\mathcal Q(\mathbf u)\\
&\quad-(\mathbf u\cdot \nabla)(\mathbf h\cdot\mathbf n)
+(\mathbf h\cdot \nabla)(\mathbf u\cdot\mathbf n)+(\mathbf n\cdot \nabla)(\mathbf u\cdot\mathbf n)-(\mathbf h\cdot\mathbf n)\mathrm{div} \mathbf u-|\mathbf n|^2\mathrm{div} \mathbf u\big)\mathrm{div}\Lambda^{k-1}\mathbf udx\\
&\leq\frac{d}{dt}\int\Lambda^{k-1}(p+\mathbf h\cdot\mathbf n)\mathrm{div}\Lambda^{k-1}\mathbf udx+(\gamma\bar P+2|\mathbf n|^2)\|\Lambda^{k}\mathbf u\|^2_0+C\|(\mathbf u,p,\mathbf h)\|_{3}\|(\mathbf u,p,\mathbf h)\|^2_{k},\\
\textrm{II}_{5}&=-\frac{d}{dt}\int(\mathbf n \cdot\nabla)\Lambda^{k-1}\mathbf u\cdot\Lambda^{k-1}\mathbf hdx+\int(\mathbf n\cdot\nabla)\Lambda^{k-1}\mathbf u \cdot\Lambda^{k-1}\mathbf h_tdx\\
&=-\frac{d}{dt}\int(\mathbf n\cdot\nabla)\Lambda^{k-1}\mathbf u\cdot\Lambda^{k-1}\mathbf hdx+\int(\mathbf n\cdot\nabla)\Lambda^{k-1}\mathbf u \cdot\Lambda^{k-1}\big(-(\mathbf u\cdot \nabla)\mathbf h
+(\mathbf h\cdot \nabla)\mathbf u+(\mathbf n\cdot \nabla)\mathbf u\\
&\quad-\mathbf h\mathrm{div} \mathbf u-\mathbf n\mathrm{div} \mathbf u\big)dx\\
&\leq-\frac{d}{dt}\int(\mathbf n \cdot\nabla)\Lambda^{k-1}\mathbf u\cdot\Lambda^{k-1}\mathbf hdx+2|\mathbf n|^2\|\Lambda^{k}\mathbf u\|^2_0+C\|(\mathbf u,\mathbf h)\|_{3}\|(\mathbf u,\mathbf h)\|^2_{k}.
\end{align*}
Then putting the estimates $\textrm{II}_{1}-\textrm{II}_{6}$ into \eqref{6-19-2}, choosing $\delta$ small enough and then using Poincar\'{e}'s inequality: $\|\Lambda^{k}\mathbf u\|^2_0\leq \|\Lambda^{k+1}\mathbf u\|^2_0,$ we obtain the desired conclusion. Thus, we complete the proof of  Lemma \ref{priori-2}.
\end{proof}
\begin{lemma}\label{priori-25-6-27-1}
Let $k\geq 3$. Under the assumptions in Lemma \ref{priori-basic-1},  then there holds
\begin{align}\label{26-4-10-1}
\frac{d}{dt}&E_{k}(t)+\frac43\|\nabla\Lambda^k\mathbf u\|_0^2+\frac43\|\sqrt{\rho}\Lambda^{k-1}\mathbf u_t\|^2_{0}\nonumber\\
&\lesssim\|\rho\|^{2(k-1)}_{3}\|\sqrt\rho\mathbf u_t\|^2_0+\|\rho\|^2_{k-1}\|\mathbf u_t\|^2_{L^\infty}+\big(\|\rho\|_3\|\mathbf u\|_3\|\mathbf u\|_{k}+\|\rho\|_k\|\mathbf u\|^2_{3}\big)\|\mathbf u\|_k\nonumber\\
&\quad+\|\rho\|^2_{k-1}\|\mathbf u\|^4_3+\|(\mathbf u,p,\mathbf h)\|^2_{3}\|(\mathbf u,p,\mathbf h)\|^2_{k}+\|(\mathbf u,p,\mathbf h)\|_{3}\|(\mathbf u,p,\mathbf h)\|^2_{k},
\end{align}
where
\begin{align*} E_k(t)&:=A_4\|\sqrt{\rho}\Lambda^k\mathbf u(\cdot,t),(\gamma\bar P)^{-\frac12}\Lambda^k p(\cdot,t),\Lambda^k\mathbf h(\cdot,t)\|^2_{0}+\mu\|\Lambda^{k}\mathbf u(\cdot,t)\|^2_0\\
&\quad+(\mu+\lambda)\|\mathrm{div}\Lambda^{k-1}\mathbf u(\cdot,t)\|_0^2-2\int\big(\Lambda^{k-1}(p+\mathbf h\cdot\mathbf n)(x,t)\mathrm{div}\Lambda^{k-1}\mathbf u(x,t)\\
&\quad-(\mathbf n\cdot\nabla)\Lambda^{k-1}\mathbf u(x,t)\cdot\Lambda^{k-1}\mathbf h(x,t)\big)dx
\end{align*}
with the positive constant $A_4$ depending only on $\mu, \mathbb T^3,|\mathbf n|,\gamma,\bar P,\underline\rho$.
Moreover,
\begin{equation}\label{26-4-10-2}
\frac54\|\Lambda^{k}(\mathbf u,p,\mathbf h)\|^2_0\leq E_k(t)\leq A_5\|\Lambda^{k}(\mathbf u,p,\mathbf h)\|^2_0.
\end{equation}
Here the   positive constants $A_5$ depends only on $\mu,\lambda,\mathbb T^{3},|\mathbf n|,\gamma,\bar P,\underline\rho,\bar\rho.$
\end{lemma}
\begin{proof} By virtue of Lemma \ref{G-N-inequality} and  \eqref{suppose-1}, we arrive at
\begin{equation}\begin{split}\label{Gagliardo-Nirenberg-1}
\|\Lambda^{k-2}\mathbf u_t\|_{0}&\leq C\|\mathbf u_t\|^{\frac{1}{k-1}}_{0}\|\Lambda^{k-1}\mathbf u_t\|_0^{\frac{k-2}{k-1}}\leq C_{k,1}\|\sqrt{\rho}\,\mathbf u_t\|^{\frac{1}{k-1}}_{0}\|\sqrt{\rho}\,\Lambda^{k-1}\mathbf u_t\|_0^{\frac{k-2}{k-1}},
\end{split}\end{equation}
where the positive constant $C_{k,1}$ depends only on $k,\mathbb T^{3},\underline{\rho}.$
It  then follows from \eqref{25-6-22-3},  \eqref{Gagliardo-Nirenberg-1} and  Young's inequality,  that
\begin{align}\label{N-shijian-1}
&\frac{d}{dt}\Big(\mu\|\Lambda^{k}\mathbf u\|^2_0
+(\mu+\lambda)\|\mathrm{div}\Lambda^{k-1}\mathbf u\|_0^2
-2\int\Lambda^{k-1}(p+\mathbf h\cdot\mathbf n)\mathrm{div}\Lambda^{k-1}\mathbf udx\nonumber\\
&\qquad+2\int(\mathbf n\cdot\nabla)\Lambda^{k-1}\mathbf u\cdot\Lambda^{k-1}\mathbf hdx\Big)\nonumber\\
&\quad+\frac32\|\sqrt{\rho}\Lambda^{k-1}\mathbf u_t\|^2_{0}\nonumber\\
&\leq C\Big(\|\rho\|^2_{k-1}\|\mathbf u_t\|^2_{L^\infty}
+\|(\mathbf u,\mathbf h)\|^2_{3}\|(\mathbf u,\mathbf h)\|^2_{k}
+\|\rho\|^2_{k-1}\|\mathbf u\|^4_3\nonumber\\
&\qquad\quad+\|(\mathbf u,p,\mathbf h)\|_{3}\|(\mathbf u,p,\mathbf h)\|^2_{k}\Big)
+A_3\|\Lambda^{k+1}\mathbf u\|^2_0\nonumber\\
&\quad+C_\delta\|\rho\|^{2(k-1)}_{3}\|\sqrt\rho\mathbf u_t\|^2_0+\delta\|\sqrt{\rho}\Lambda^{k-1}\mathbf u_t\|_0^2.
\end{align}
Choosing a sufficiently large number $A_4$ such that
\begin{equation}\label{A4-1}A_4\geq\frac{A_3+2}{\mu},\end{equation}
 and then  multiplying inequality \eqref{25-6-22-1} by $A_4$,  using \eqref{Gagliardo-Nirenberg-1} and  Young's inequality, we infer that
\begin{align*}
A_4&\frac{d}{dt}\|(\sqrt{\rho}\Lambda^k\mathbf u,(\gamma\bar P)^{-\frac12}\Lambda^k p,\Lambda^k\mathbf h)\|^2_{0}+(A_3+2)\|\nabla\Lambda^k\mathbf u\|_0^2\\
&\lesssim\|\rho\|^2_{3}\|\Lambda^{k-2}\mathbf u_t\|^2_{0}+\|\rho\|^2_{k-1}\|\mathbf u_t\|^2_{L^\infty}+\big(\|\rho\|_3\|\mathbf u\|_3\|\mathbf u\|_{k}+\|\rho\|_k\|\mathbf u\|^2_{3}\big)\|\mathbf u\|_k\\
&\quad+\|(\mathbf u,p,\mathbf h)\|^2_{3}\|(\mathbf u,p,\mathbf h)\|^2_{k}+\|(\mathbf u,p,\mathbf h)\|_{3}\|(\mathbf u,p,\mathbf h)\|^2_{k}\\
&\leq C_\delta\|\rho\|^{2(k-1)}_{3}\|\sqrt\rho\mathbf u_t\|^2_0+C\|\rho\|^2_{k-1}\|\mathbf u_t\|^2_{L^\infty}+C(\|\rho\|_3\|\mathbf u\|_3\|\mathbf u\|_{k}+\|\rho\|_k\|\mathbf u\|^2_{3})\|\mathbf u\|_k\\
&\quad+C\|(\mathbf u,p,\mathbf h)\|^2_{3}\|(\mathbf u,p,\mathbf h)\|^2_{k}+C\|(\mathbf u,p,\mathbf h)\|_{3}\|(\mathbf u,p,\mathbf h)\|^2_{k}+\delta\|\sqrt{\rho}\Lambda^{k-1}\mathbf u_t\|_0^2.
\end{align*}
Then substituting the above inequality to \eqref{N-shijian-1} and then choosing $\delta$ small enough, we conclude that \eqref{26-4-10-1} holds.

Next, we shall prove \eqref{26-4-10-2}.
It follows from  H\"{o}lder's and  Young's inequalities, Poincar\'{e}'s inequality, and \eqref{suppose-1},  that
\begin{align*}
&\Big|\int\big(\Lambda^{k-1}(p+\mathbf h\cdot\mathbf n)\mathrm{div}\Lambda^{k-1}\mathbf u
-(\mathbf n\cdot\nabla)\Lambda^{k-1}\mathbf u\cdot\Lambda^{k-1}\mathbf h\big)dx\Big|\\
&\leq\big(\|\Lambda^{k-1}p\|_0+2|\mathbf n|\|\Lambda^{k-1}\mathbf h\|_0\big)\|\Lambda^{k}\mathbf u\|_0
\\&\leq \big(4|\mathbf n|^2+1\big)\|(\Lambda^{k-1}p,\Lambda^{k-1}\mathbf h)\|^2_0+\|\Lambda^{k}\mathbf u\|^2_0
\\&\leq \big(4|\mathbf n|^2+1\big)\|(\Lambda^{k}\mathbf u,\Lambda^{k}p,\Lambda^{k}\mathbf h)\|^2_0,
\end{align*}
 and
\begin{align*}&\min\{\frac{\underline\rho}{2},(\gamma\bar P)^{-1},1\}\|(\Lambda^{k}\mathbf u,\Lambda^{k}p,\Lambda^{k}\mathbf h)\|^2_0\\&\leq\|(\sqrt{\rho}\Lambda^k\mathbf u(\cdot,t),(\gamma\bar P)^{-\frac12}\Lambda^k p(\cdot,t),\Lambda^k\mathbf h(\cdot,t))\|^2_{0}\\
&\leq\max\{2\bar\rho,(\gamma\bar P)^{-1},1\}\|(\Lambda^{k}\mathbf u,\Lambda^{k}p,\Lambda^{k}\mathbf h)\|^2_0.\end{align*}
Therefore, recalling the definition of $E_k(t)$, we can choose a sufficiently large constant
$A_4$  satisfying \eqref{A4-1} such that
$$\frac54\|(\Lambda^{k}\mathbf u,\Lambda^{k}p,\Lambda^{k}\mathbf h)\|^2_0\leq E_k(t)\leq A_5\|(\Lambda^{k}\mathbf u,\Lambda^{k}p,\Lambda^{k}\mathbf h)\|^2_0,$$
 where the positive constant $A_5$ depends only on $A_4,\bar\rho,\gamma,\bar P,\mu,\lambda,|\mathbf n|.$ This leads to \eqref{26-4-10-2}. Thus, we complete the proof of Lemma \ref{priori-25-6-27-1}.
\end{proof}
In the following three lemmas, we shall construct the higher-order ($M-$th order) energy estimates for $(\mathbf u,p,\mathbf h)$ under the \emph{a priori} assumptions together with the corresponding decay estimates. For that, we first exploit the coupling among the velocity equation, the pressure equation, and the magnetic field equation to develop the dissipation estimates for the magnetic field and pressure.
\begin{lemma}\label{priori-3}
Let $M\geq r+1 $. Under the assumptions in Lemma \ref{priori-basic-1},  then there holds
\begin{align}\label{26-4-30-1}
\|(p,\mathbf h)\|^2_{M-1-r}&\lesssim\|\sqrt{\rho}\Lambda^{M-1}\mathbf u_t\|^2_{0}+\|\sqrt{\rho}\mathbf u_t\|^2_{0}+\|\rho\|^2_{M-1}\|\mathbf u_t\|^2_{L^\infty}\nonumber\\
&\quad+(1+\|\rho\|^2_M)\|(\mathbf u,\mathbf h)\|^4_M+\|\Lambda^{M+1}\mathbf u\|_0^2.
\end{align}
\end{lemma}
\begin{proof}First, taking the inner product of the equation $\eqref{no-heat-mhd3}_2$ with the constant vector $\mathbf n$ and using $(\mathbf n\cdot\nabla)(\mathbf h\cdot\mathbf n)=(\mathbf n\cdot\nabla)\mathbf h\cdot\mathbf n$, we obtain
$$(\mathbf n\cdot\nabla) p=-\rho\left(\partial_{t}\mathbf u\cdot\mathbf n+(\mathbf u\cdot\nabla)\mathbf u\cdot\mathbf n\right)+\mu\Delta \mathbf u\cdot\mathbf n+(\mu+\lambda)(\mathbf n\cdot\nabla)\mathrm{div}\mathbf u-\frac12(\mathbf n\cdot\nabla)|\mathbf h|^2+(\mathbf h\cdot\nabla)\mathbf h\cdot\mathbf n.$$
This along with Lemma \ref{Hk-estimate6-20}, \eqref{suppose-1} and   H\"{o}lder's inequality, Sobolev's embedding inequality and Gagliardo-Nirenberg interpolation inequality  gives rise to
\begin{align}
\|(\mathbf n\cdot\nabla)\Lambda^{M-1}p\|^2_0&\leq8|\mathbf n|^2\big(\|(\rho\mathbf u_t,\rho(\mathbf u\cdot\nabla)\mathbf u,\nabla(|\mathbf h|^2),(\mathbf h\cdot\nabla)\mathbf h)\|^2_{M-1}+(2\mu+\lambda)^2\|\nabla\mathbf u\|_M^2\big)\nonumber\\
&\lesssim\|\rho\|^2_{L^\infty}\|\mathbf u_t\|^2_{M-1}+\|\rho\|^2_{M-1}\|\mathbf u_t\|^2_{L^\infty}+\|\rho\|^2_{L^\infty}(\|\mathbf u\|^2_{L^\infty}\|\mathbf u\|^2_M+\|\mathbf u\|^2_{M-1}\|\nabla\mathbf u\|^2_{L^\infty})\nonumber\\
&\quad+\|\rho\|^2_{M-1}\|\mathbf u\|^2_{L^\infty}\|\nabla\mathbf u\|^2_{L^\infty}+\|\mathbf h,\nabla\mathbf h\|^2_{L^\infty}\|\mathbf h,\nabla\mathbf h\|^2_{M-1}+\|\Lambda^{M+1}\mathbf u\|_0^2\nonumber\\
&\lesssim\|\sqrt{\rho}\Lambda^{M-1}\mathbf u_t\|^2_{0}+\|\sqrt{\rho}\mathbf u_t\|^2_{0}+\|\rho\|^2_{M-1}\|\mathbf u_t\|^2_{L^\infty}\nonumber\\
&\quad+(1+\|\rho\|^2_M)\|\mathbf u\|^4_M+\|\mathbf h\|^4_M+\|\Lambda^{M+1}\mathbf u\|_0^2.\label{2026-4-30-2}
\end{align}
It follows from Lemma \ref{Diophantine-inequality} and Poincar\'{e}'s inequality that
$$\big\|p-|\mathbb T^3|^{-1}\int pdx\big\|_{M-1-r}\lesssim\|(\mathbf n\cdot\nabla)\Lambda^{r}p\|_0+\|(\mathbf n\cdot\nabla)\Lambda^{M-1}p\|_0\lesssim\|(\mathbf n\cdot\nabla)\Lambda^{M-1}p\|_0,$$
which together with \eqref{basic-2} and \eqref{2026-4-30-2},  yields
\begin{align}
\|p\|^2_{M-1-r}&\lesssim\big|\int pdx\big|^2+\|(\mathbf n\cdot\nabla)\Lambda^{M-1}p\|^2_0\nonumber\\
&\lesssim\sup_{\tau\in[0,T]}\|\rho\|^2_{L^\infty}\|(\mathbf u,\mathbf h)\|^2_{0}\|(\nabla\mathbf u,\nabla\mathbf h)\|^2_0+\|(\mathbf n\cdot\nabla)\Lambda^{M-1}p\|^2_0\nonumber\\
&\lesssim \|(\mathbf u,\mathbf h)\|^2_{0}\|(\nabla\mathbf u,\nabla\mathbf h)\|^2_0+\|\sqrt{\rho}\Lambda^{M-1}\mathbf u_t\|^2_{0}+\|\sqrt{\rho}\mathbf u_t\|^2_{0}+\|\rho\|^2_{M-1}\|\mathbf u_t\|^2_{L^\infty}\nonumber\\
&\quad+(1+\|\rho\|^2_M)\|\mathbf u\|^4_M+\|\mathbf h\|^4_M+\|\Lambda^{M+1}\mathbf u\|_0^2.\label{2026-6-4-2}
\end{align}
 Similar to the derivation of  \eqref {2026-4-30-2}, it follows from $\eqref{no-heat-mhd3}_2$, $\mathrm{div} \mathbf h=0$, Lemma \ref{Hk-estimate6-20}, \eqref{suppose-1} and  H\"{o}lder's inequality, Sobolev's embedding inequality and
Gagliardo-Nirenberg interpolation inequality that
\begin{align*}
&\big\|(\mathbf n\cdot\nabla)\Lambda^{M-1}\mathbf h\big\|^2_{0}+\big\|\nabla\Lambda^{M-1}(p+\frac12|\mathbf h|^2+\mathbf h\cdot\mathbf n)\big\|^2_0\\&=\big\|\Lambda^{M-1}\big((\mathbf n\cdot\nabla)\mathbf h-\nabla(p+\frac12|\mathbf h|^2+\mathbf h\cdot\mathbf n)\big)\big\|^2_{0}\\
&\lesssim\|\sqrt{\rho}\Lambda^{M-1}\mathbf u_t\|^2_{0}+\|\sqrt{\rho}\mathbf u_t\|^2_{0}+\|\rho\|^2_{M-1}\|\mathbf u_t\|^2_{L^\infty}\\
&\quad+(1+\|\rho\|^2_M)\|\mathbf u\|^4_M+\|\mathbf h\|^4_M+\|\Lambda^{M+1}\mathbf u\|_0^2.
\end{align*}
Thus,  by virtue of the above inequality, Lemma \ref{Diophantine-inequality}, \eqref{basic-2} and \eqref{2026-6-4-2},  we  finally conclude that \eqref{26-4-30-1} holds. This completes the proof of Lemma \ref{priori-3}.
\end{proof}
\begin{lemma}\label{priori-combine}
 Let $M\geq 3$ and $M_{\rho}:=\|\rho_0\|^{{\frac{4M-4}{2M-5}}}_{M}+\|\rho_0\|^{2(M-1)}_{M}+1.$ Assume $(\rho,\mathbf u,p,\mathbf h)$ is a smooth solution to   the system \eqref{no-heat-mhd3} on $\mathbb T^3\times[0,T]$ for some $T>0$ satisfying \eqref{suppose-1} and
 \begin{equation}\label{rhoH-M}
 \sup_{t\in[0,T]}\|\rho\|_{M}\leq4\|\rho_0\|_M.
\end{equation}
Then, there is a positive constant $A_6$ satisfying $A_6\sim M_\rho$ such that
\begin{align}\label{26-4-12-2}\frac{d}{dt}&\mathcal E(t)+\frac{1}{4}\underline\rho\|\mathbf u_t\|^2_{M-1}+\frac12\|\mathbf u\|_{M+1}^2+\delta_0\|(p,\mathbf h)\|^2_{M-1-r}\nonumber\\
&\lesssim (1+\|(\mathbf u,p,\mathbf h)\|_2)\|(\mathbf u,p,\mathbf h)\|_2^3+\|\rho_0\|_M\|\mathbf u\|^3_{M}+\|\rho_0\|^2_{M}\|\mathbf u\|^4_M\nonumber\\
&\quad+\|(\mathbf u,p,\mathbf h)\|^4_{M}+\|(\mathbf u,p,\mathbf h)\|^3_{M}+(1+\|\rho_0\|^2_M)\|(\mathbf u,\mathbf h)\|^4_M,
\end{align}
and
\begin{equation}\label{26-4-12-3}
\|(\mathbf u,p,\mathbf h)\|^2_M\leq \mathcal E(t)\leq B_4M_{\rho}\|(\mathbf u,p,\mathbf h)\|^2_M
\end{equation}
with $\mathcal E(t):=A_6\mathcal A(t)+E_{M}(t).$  Here the positive constants $B_4$ and $\delta_0\in(0,1)$ depend only on $\mathbb T^3,|\mathbf n|,$ $M,\mu,\lambda,\gamma,\bar P,\underline{\rho}^{-1},\bar\rho,M_0,r,c.$
\end{lemma}
\begin{proof}
It follows from $\eqref{no-heat-mhd3}_{1}$,  \eqref{p-poincare-AAA},  \eqref{suppose-1} and H\"{o}lder's inequality that
\begin{align}\label{26-4-14-1A}\Big|\int\rho\mathbf u_tdx\Big|&=\Big|\frac{d}{dt}\int\rho\mathbf udx-\int\rho_t\mathbf udx\Big|\nonumber\\
&=\Big|\int\mathrm{div}(\rho\mathbf u)\mathbf udx\Big|\nonumber\\
&=\Big|\int(\rho\mathbf u\cdot\nabla)\mathbf udx\Big|\nonumber\\
&\leq \bar\rho\|\mathbf u\|_0\|\nabla\mathbf u\|_0,\end{align}
which together with  Corollary \ref{jiaquanG-N}, \eqref{suppose-1}, \eqref{26-4-14-1A} and Sobolev's embedding inequality implies  \begin{align}\label{26-4-14-1}\|\mathbf u_t\|_{L^\infty}&\lesssim\|\mathbf u_t\|^{\frac{2M-5}{2M-2}}_0\|\Lambda^{M-1}\mathbf u_t\|^{\frac{3}{2M-2}}_0+\Big|\int\rho\mathbf u_tdx\Big|\nonumber\\
&\lesssim\|\sqrt{\rho}\mathbf u_t\|^{\frac{2M-5}{2M-2}}_0\|\sqrt{\rho}\Lambda^{M-1}\mathbf u_t\|^{\frac{3}{2M-2}}_0+\|\mathbf u\|^2_{M}.
\end{align}
Choosing $k=M$ in \eqref{26-4-10-1} and using \eqref{26-4-14-1}, \eqref{rhoH-M} and Young's inequality, we infer that
\begin{align}\label{26-4-12-1}
&\frac{d}{dt}E_{M}(t)+\frac43\|\nabla\Lambda^M\mathbf u\|_0^2+\frac43\|\sqrt{\rho}\Lambda^{M-1}\mathbf u_t\|^2_{0}\nonumber\\
&\leq C_{\delta_2}(\|\rho_0\|^{2(M-1)}_{M}+\|\rho_0\|^{{\frac{4M-4}{2M-5}}}_{M})\|\sqrt\rho\mathbf u_t\|^2_0+\delta_2\|\sqrt{\rho}\Lambda^{M-1}\mathbf u_t\|^2_0+C\|\rho_0\|_M\|\mathbf u\|^3_{M}\nonumber\\
&\quad+C\|\rho_0\|^2_{M}\|\mathbf u\|^4_M+C\|(\mathbf u,p,\mathbf h)\|^2_{3}\|(\mathbf u,p,\mathbf h)\|^2_{M}+C\|(\mathbf u,p,\mathbf h)\|_{3}\|(\mathbf u,p,\mathbf h)\|^2_{M}.
\end{align}
Moreover, it follows from
 \eqref{26-4-30-1}, \eqref{rhoH-M}, \eqref{26-4-14-1} and  Young's inequality that
\begin{align}\label{2026-4-30-4}
\|(p,\mathbf h)\|^2_{M-1-r}
&\leq C_2(1+\|\rho_0\|^{{\frac{4M-4}{2M-5}}}_{M})\|\sqrt\rho\mathbf u_t\|_0^2+C_2\|\sqrt{\rho}\Lambda^{M-1}\mathbf u_t\|^{2}_0\nonumber\\
&\quad+C_2(1+\|\rho_0\|^2_M)\|(\mathbf u,\mathbf h)\|^4_M+C_2\|\Lambda^{M+1}\mathbf u\|_0^2,
\end{align}
where the positive constant $C_2$ depends only $\mathbb T^3,|\mathbf n|,M,\mu,\lambda,\gamma,\bar P,\underline{\rho}^{-1},\bar\rho,r,c$. Taking $\delta_2=\frac1{24}$ and choosing two positive constants $A_6,\delta_0$ satisfying  $A_6\geq2(C_{\delta_2}+C_2)(1+\|\rho_0\|^{2(M-1)}_{M}+\|\rho_0\|^{{\frac{4M-4}{2M-5}}}_{M})$ and $\delta_0=\min\{(24C_2)^{-1},1\},$   multiplying \eqref{basic-3} by $A_6$ and \eqref{2026-4-30-4} by $\delta_0$,  and then adding the results to \eqref{26-4-12-1}, we obtain
\begin{align*}
&\frac{d}{dt}\big(A_6\mathcal A(t)+E_M(t)\big)+\Big(A_6-(C_{\delta_2}+C_2)(1+\|\rho_0\|^{2(M-1)}_{M}+\|\rho_0\|^{{\frac{4M-4}{2M-5}}}_{M})\Big)\|\mathbf u_t\|_0^2+A_6\|\nabla\mathbf u\|_0^2\\
&\quad +(\frac43-\delta_0C_2)\|\nabla\Lambda^M\mathbf u\|_0^2+(\frac43-\delta_0C_2-\delta_2)\|\sqrt{\rho}\Lambda^{M-1}\mathbf u_t\|^2_{0}+\delta_0\|(p,\mathbf h)\|^2_{M-1-r}\\
&\lesssim(1+\|(\mathbf u,p,\mathbf h)\|_2)\|(\mathbf u,p,\mathbf h)\|_2^3+\|\rho_0\|_M\|\mathbf u\|^3_{M}+\|\rho_0\|^2_{M}\|\mathbf u\|^4_M\\
&\quad+\|(\mathbf u,p,\mathbf h)\|^4_{M}+\|(\mathbf u,p,\mathbf h)\|^3_{M}+(1+\|\rho_0\|^2_M)\|(\mathbf u,\mathbf h)\|^4_M.
\end{align*}
This along with \eqref{suppose-1}, yields
\begin{align*}
&\frac{d}{dt}\big(A_6\mathcal A(t)+E_M(t)\big)+\frac{1}{2}A_6\|\nabla\mathbf u\|_0^2+\frac54\|\nabla\Lambda^M\mathbf u\|_0^2+\frac{1}{4}A_6\underline\rho\|\mathbf u_t\|_0^2\\&\quad+\frac{5}{8}\underline\rho\|\sqrt{\rho}\Lambda^{M-1}\mathbf u_t\|^2_{0}+\delta_0\|(p,\mathbf h)\|^2_{M-1-r}\\
&\lesssim(1+\|(\mathbf u,p,\mathbf h)\|_2)\|(\mathbf u,p,\mathbf h)\|_2^3+\|\rho_0\|_M\|\mathbf u\|^3_{M}+\|\rho_0\|^2_{M}\|\mathbf u\|^4_M\\
&\quad+\|(\mathbf u,p,\mathbf h)\|^4_{M}+\|(\mathbf u,p,\mathbf h)\|^3_{M}+(1+\|\rho_0\|^2_M)\|(\mathbf u,\mathbf h)\|^4_M.
\end{align*}
Thanks to  the following fact \begin{equation}\label{2026-4-30-5}\|f\|^2_0+\|\Lambda^kf\|^2_0\leq\|f\|^2_k\leq C_{k,2}\|f\|^2_0+\frac54\|\Lambda^k f\|^2_0\end{equation} with a positive constants $C_{k,2}$ depending only on $k,\mathbb T^{3},$ the first inequality stated in \eqref{basic-2} and \eqref{suppose-1}, we have
\begin{equation}\label{2026-6-5-1}\|\mathbf u\|^2_{M+1}\leq C_3\|\nabla\mathbf u\|^2_0+\frac54\|\Lambda^{M+1}\mathbf u\|^2_0,\end{equation} where the positive constant $C_3$ depends only on $M,\mathbb T^{3},\gamma,M_0,\bar\rho.$  Thus, we can again choose a larger constant $A_6$ satisfying  $A_6\geq2(C_{\delta_2}+C_2)(1+\|\rho_0\|^{2(M-1)}_{M}+\|\rho_0\|^{{\frac{4M-4}{2M-5}}}_{M})$  and $A_6\sim\|\rho_0\|^{{\frac{4M-4}{2M-5}}}_{M}+\|\rho_0\|^{2(M-1)}_{M}+1$ such that
\begin{align}\label{2026-6-5-2}
\frac{d}{dt}&\big(A_6\mathcal A(t)+E_M(t)\big)+\|\mathbf u\|_{M+1}^2+\frac{\underline\rho}{2}\|\mathbf u_t\|^2_{M-1}+\delta_0\|(p,\mathbf h)\|^2_{M-1-r}\nonumber\\
&\lesssim (1+\|(\mathbf u,p,\mathbf h)\|_2)\|(\mathbf u,p,\mathbf h)\|_2^3+\|\rho_0\|_M\|\mathbf u\|^3_{M}+\|\rho_0\|^2_{M}\|\mathbf u\|^4_M\nonumber\\
&\quad+\|(\mathbf u,p,\mathbf h)\|^4_{M}+\|(\mathbf u,p,\mathbf h)\|^3_{M}+(1+\|\rho_0\|^2_M)\|(\mathbf u,\mathbf h)\|^4_M,
\end{align}
which leads to \eqref{26-4-12-2}. Moreover, \eqref{26-4-12-3} follows from the definition of $A_6$ and \eqref{26-4-10-2}. This completes the proof of Lemma \ref{priori-combine}.
\end{proof}
In what follows, we need to exploit some  decay estimates for $\|(\mathbf u,p,\mathbf h)\|_M.$
\begin{lemma}\label{priori-4}
Assume that all the conditions of Lemma \ref{priori-combine} hold, and that the integers $M,N$ satisfy $N\geq M+2(r+1)$. Then there is a positive constant $\epsilon_0<1$, depending only on $\mathbb T^3$, $|\mathbf n|$, $M$, $N$, $\mu$, $\lambda$, $\gamma$, $\bar P$, $\underline{\rho}^{-1}$, $\bar\rho$, $M_0$, $r$ and $c$, with the following property. If
\begin{align}\label{2026-5-1-1}
 (1+\|\rho_0\|_M)\sup_{t\in[0,T]}\|(\mathbf u,p,\mathbf h)\|_{N}^{\frac{2(r+1)}{N-M+r+1}}+ M_{\rho}^{\frac{N-M+r+1}{N-M}}\sup_{t\in[0,T]}\|(\mathbf u,p,\mathbf h)\|_{N}^{\frac{2(r+1)}{N-M}}\leq\eta_0\end{align}
holds for some $\eta_0\in(0,\epsilon_0]$, then the following estimates hold:
\begin{equation}\begin{split}\label{25-6-27-3}
\|(\mathbf u,p,\mathbf h)\|^2_M\leq \Big(\frac{r+1}{N-M}t+\big(B_4M_\rho\|(\mathbf u_0,p_0,\mathbf h_0)\|^2_M\big)^{-\frac{r+1}{N-M}}\Big)^{-\frac{N-M}{r+1}},
\end{split}\end{equation}
and
\begin{align}\label{25-6-27-4}
\sup_{t\in[0,T]}&\beta(t)\|(\mathbf u,p,\mathbf h)\|^2_M+\int_0^T\beta(t)\Big(\frac{1}{2}\underline\rho\|\mathbf u_t\|^2_{M-1}+\frac12\|\mathbf u\|_{M+1}^2+\frac{1}{2}\delta_0\|(p,\mathbf h)\|^2_{M-1-r}\Big)dt\nonumber\\
&\lesssim M_\rho\|(\mathbf u_0,p_0,\mathbf h_0)\|^2_M,
\end{align}
where  $\beta(t):=\big(\frac{r+1}{N-M}\|(\mathbf u_0,p_0,\mathbf h_0)\|_M^{\frac{2r+2}{N-M}}t+1\big)^{\frac{N-M-1}{r+1}}$  and the  two positive  constants  $B_4,\delta_0$ are defined as in Lemma \ref{priori-combine}.
\end{lemma}
\begin{proof}
Employing  \eqref{26-4-12-2}, \eqref{2026-5-1-1}, $\frac{3(N-M)}{N-M+r+1}\geq2$ and the  Gagliardo-Nirenberg interpolation inequality,
\begin{equation}\label{GN-2026-1}
\|\cdot\|_M\lesssim\|\cdot\|_{M-r-1}^{\frac{N-M}{N-M+r+1}}\|\cdot\|_{N}^{\frac{r+1}{N-M+r+1}},
\end{equation}
 we conclude that
\begin{align*}&\frac{d}{dt}\mathcal E(t)+\frac{1}{2}\underline\rho\|\mathbf u_t\|^2_{M-1}+\|\mathbf u\|_{M+1}^2+\delta_0\|(p,\mathbf h)\|^2_{M-1-r}\\
&\leq C(1+\|(\mathbf u,p,\mathbf h)\|_M)\|(\mathbf u,p,\mathbf h)\|_{M-r-1}^{\frac{3(N-M)}{N-M+r+1}}\|(\mathbf u,p,\mathbf h)\|_{N}^{\frac{3(r+1)}{N-M+r+1}}+C\|\rho_0\|_M\|\mathbf u\|_{M-r-1}^{\frac{3(N-M)}{N-M+r+1}}\|\mathbf u\|_{N}^{\frac{3(r+1)}{N-M+r+1}}\\
&\quad+C(1+\|\rho_0\|^2_M)\|(\mathbf u,p,\mathbf h)\|_{M-r-1}^{\frac{4(N-M)}{N-M+r+1}}\|(\mathbf u,p,\mathbf h)\|_{N}^{\frac{4(r+1)}{N-M+r+1}}\\
&\leq 3C(1+\epsilon_0)\epsilon_0\|(\mathbf u,p,\mathbf h)\|_{M-r-1}^2.
\end{align*}
Choosing small enough $\epsilon_0$  such that  $3C(1+\epsilon_0)\epsilon_0\leq\min\{\frac12,\frac{\delta_0}{2}\}$, implies  that
\begin{align}\label{2026-5-3-1}\frac{d}{dt}\mathcal E(t)+\frac{1}{2}\underline\rho\|\mathbf u_t\|^2_{M-1}+\frac12\|\mathbf u\|_{M+1}^2+\frac{1}{2}\delta_0\|(p,\mathbf h)\|^2_{M-1-r}\leq0.
\end{align}
According to \eqref{26-4-12-3}, \eqref{2026-5-1-1} and \eqref{GN-2026-1}, we deduce  that
\begin{align*}
\mathcal E^{\frac{N-M+r+1}{N-M}}(t)&\lesssim M_\rho^{\frac{N-M+r+1}{N-M}}\|(\mathbf u,p,\mathbf h)\|^{\frac{2(N-M+r+1)}{N-M}}_M\\
&\lesssim M_\rho^{\frac{N-M+r+1}{N-M}}\|(\mathbf u,p,\mathbf h)\|^{\frac{2(r+1)}{N-M}}_N\|(\mathbf u,p,\mathbf h)\|^{2}_{M-r-1}\\
&\leq C\epsilon_0\|(\mathbf u,p,\mathbf h)\|^{2}_{M-r-1},
\end{align*}
which together with \eqref{2026-5-3-1} yields the following  Lyapunov-type inequality,
$$\frac{d}{dt}\mathcal E(t)+\mathcal E^{\frac{N-M+r+1}{N-M}}(t)\leq0.$$
This leads to
\begin{equation}\label{2026-5-5-1}\mathcal E(t)\leq\mathcal E(0)\Big(\frac{r+1}{N-M}\mathcal E^{\frac{r+1}{N-M}}(0)t+1\Big)^{-\frac{N-M}{r+1}}.\end{equation}
Moreover, combining  with \eqref{26-4-12-3}, we obtain  \eqref{25-6-27-3}.

It remains to prove \eqref{25-6-27-4}. Multiplying the inequality \eqref{2026-5-3-1} by $$\alpha(t):=\Big(\frac{r+1}{N-M}\mathcal E^{\frac{r+1}{N-M}}(0)t+1\Big)^{\frac{N-M-1}{r+1}},$$ yields
\begin{equation*}\begin{split}
\frac{d}{dt}&\Big(\alpha(t)\mathcal E(t)\Big)+\alpha(t)\Big(\frac{1}{2}\underline\rho\|\mathbf u_t\|^2_{M-1}+\frac12\|\mathbf u\|_{M+1}^2+\frac{1}{2}\delta_0\|(p,\mathbf h)\|^2_{M-1-r}\Big)\\
&\leq\frac{N-M-1}{N-M}\mathcal E^{\frac{r+1}{N-M}}(0)\Big(\frac{r+1}{N-M}\mathcal E^{\frac{r+1}{N-M}}(0)t+1\Big)^{\frac{N-M-r-2}{r+1}}\mathcal E(t).
\end{split}\end{equation*}
Then, by integrating on the time $t$ and  using \eqref{2026-5-5-1}, we conclude that
\begin{align*}\sup_{t\in[0,T]}&\alpha(t)\mathcal E(t)+\int_0^T\alpha(t)\Big(\frac{1}{2}\underline\rho\|\mathbf u_t\|^2_{M-1}+\frac12\|\mathbf u\|_{M+1}^2+\frac{\delta_0}{2}\|(p,\mathbf h)\|^2_{M-1-r}\Big)dt\\
&\leq\mathcal E(0)+\frac{N-M-1}{N-M}\mathcal E^{\frac{N-M+r+1}{N-M}}(0)\int_0^T\Big(\frac{r+1}{N-M}\mathcal  E^{\frac{r+1}{N-M}}(0)t+1\Big)^{\frac{-r-2}{r+1}}dt\\&\lesssim\mathcal E(0),
\end{align*}
which together with \eqref{26-4-12-3} implies   \eqref{25-6-27-4}. This completes the proof of Lemma \ref{priori-4}.\end{proof}
 Based on the decay estimates for $\|(\mathbf u,p,\mathbf h)\|_M$ in Lemma \ref{priori-4}, we shall construct some estimates for $\rho$ including  a linear growth estimate on time $t$ of  $\|\rho\|_N$.
\begin{lemma}\label{priori-25-6-27-2}
Assume that all the conditions of Lemma \ref{priori-4} hold. Then there is a positive constant $\epsilon_1<1$, depending only on $\mathbb T^3$, $|\mathbf n|$, $M$, $N$, $\mu$, $\lambda$, $\gamma$, $\bar P$, $\underline{\rho}^{-1}$, $\bar\rho$, $M_0$, $r$ and $c$, with the following property. If
\begin{equation}\label{2026-5-6-4}
M_\rho^{\frac12}\|(\mathbf u_0,p_0,\mathbf h_0)\|_M^{\frac{N-M-r-1}{N-M}}\leq \eta_1,
\end{equation}
holds for some $\eta_1\in(0,\epsilon_1]$, then for every $t\in[0,T]$ there hold
\begin{equation}\label{2026-5-6-1}
\sup_{t\in[0,T]}\|\rho\|^2_{M}\leq 2\|\rho_0\|_{M}^2,
\end{equation}
\begin{equation}\label{2026-5-7-4}
\frac{2}{3}\underline{\rho}\leq\rho\leq\frac32\bar\rho,\ \ \text{in}\ \mathbb T^3\times[0,T],
\end{equation}
\begin{align}\label{2026-5-6-2}
\sup_{\tau\in[0,t]}\|\rho\|^2_{N}\leq A_7\Big(1+t\int_0^t\|\mathbf u\|_{N+1}^2d\tau \Big),
\end{align}
 where $A_7\sim\|\rho_0\|^2_{N}.$
\end{lemma}
\begin{proof}
Given $L\geq3,$ applying $\Lambda^l$ with $0\leq l\leq L$ to  the equation $\eqref{no-heat-mhd3}_{1}$,  taking $L^2$-scalar product
with $\Lambda^l\rho$, integrating by parts, and  then using Lemmas \ref{Hk-estimate6-20}-\ref{Hk-estimate} and Sobolev's embedding inequality, we infer that
\begin{align}\label{25-6-27-5}
\frac{d}{dt}\|\Lambda^l\rho\|_0^2&=\int|\Lambda^l\rho|^2\mathrm{div}\mathbf udx-2\int[\Lambda^l,\mathbf u\cdot\nabla]\rho\Lambda^l\rho dx-2\int\Lambda^l(\rho\mathrm{div}\mathbf u)\Lambda^l\rho dx\nonumber\\
&\lesssim\|\mathbf u\|_{3}\|\rho\|^2_L+\|\rho\|_3\|\mathbf u\|_{L+1}\|\rho\|_L,
\end{align}
where we have used the following equality
$$\int (\mathbf u\cdot\nabla\Lambda^l\rho)\Lambda^l\rho\,dx=-\frac12\int|\Lambda^l\rho|^2\mathrm{div}\mathbf u\,dx.$$
Then choosing $L=M$ and summing up for any $0\leq l\leq M$, we conclude that
\begin{equation}\begin{split}\label{25-6-27-6}
\frac{d}{dt}\|\rho\|_{M}^2\leq C\|\mathbf u\|_{M+1}\|\rho\|^2_{M},
\end{split}\end{equation}
which together with  Gronwall's inequality,   \eqref{25-6-27-4}, \eqref{2026-5-6-4} and H\"{o}lder's inequality,  yields
\begin{align*}
\sup_{t\in[0,T]}\|\rho\|_{M}^2&\leq e^{C\big(\int_0^T\beta(t)\|\mathbf u\|^2_{M+1}dt\big)^{\frac12}\big(\int_0^\infty\beta^{-1}(t)dt\big)^{\frac12}}\|\rho_0\|_{M}^2\\
&\leq e^{CM_\rho^{\frac{1}{2}}\|(\mathbf u_0,p_0,\mathbf h_0)\|_M^{\frac{N-M-r-1}{N-M}}}\|\rho_0\|_{M}^2\\
&\leq e^{C\epsilon_1}\|\rho_0\|_{M}^2,
\end{align*}
where we have used the following fact
$$\int_0^\infty\beta^{-1}(t)dt=\frac{N-M}{r+1}\|(\mathbf u_0,p_0,\mathbf h_0)\|_M^{\frac{-2r-2}{N-M}}\int_0^\infty(s+1)^{\frac{M+1-N}{r+1}}ds\leq C\|(\mathbf u_0,p_0,\mathbf h_0)\|_M^{\frac{-2r-2}{N-M}},$$
due to $\frac{M+1-N}{r+1}<-1.$ Thus,  by choosing $\epsilon_1$ small enough, we deduce that  \eqref{2026-5-6-1} holds for all $t\in[0,T]$.

On the other hand, it follows from  $\eqref{no-heat-mhd3}_1$, H\"{o}lder's inequality, Sobolev's embedding inequality, \eqref{25-6-27-4} and \eqref{2026-5-6-4}, that
\begin{align*}
\sup_{(x,t)\in\mathbb T^3\times[0,T]}\rho&\leq\|\rho_0\|_{L^\infty}e^{\int_0^T\|\mathrm{div}\mathbf u\|_{L^\infty}dt}\leq\bar\rho e^{\int_0^T\|\mathrm{div}\mathbf u\|_{L^\infty}dt}\\
&\leq \bar\rho e^{C\left(\int_0^T\beta(t)\|\mathbf u\|^2_{M+1}dt\right)^{\frac12}\left(\int_0^\infty\beta^{-1}(t)dt\right)^{\frac12}}\leq \bar\rho e^{C\epsilon_1},\\
\inf_{(x,t)\in\mathbb T^3\times[0,T]}\rho&\geq\underline\rho e^{-\int_0^T\|\mathrm{div}\mathbf u\|_{L^\infty}dt}\\
&\geq \underline\rho e^{-C\left(\int_0^T\beta(t)\|\mathbf u\|^2_{M+1}dt\right)^{\frac12}\left(\int_0^\infty\beta^{-1}(t)dt\right)^{\frac12}}\geq \underline\rho e^{-C\epsilon_1}.
\end{align*}
Thus, \eqref{2026-5-7-4} follows by choosing $\epsilon_1$ small enough in the above two estimates.

It remains to bound \eqref{2026-5-6-2}. By choosing $L=N$ and  then summing up for any $0\leq l\leq N$, we also deduce that
\begin{equation*}
\frac{d}{dt}\|\rho\|_{N}\lesssim\|\mathbf u\|_{M}\|\rho\|_{N}+\|\mathbf u\|_{N+1}\|\rho\|_{M}.
\end{equation*}
Applying  Gronwall's inequality and using \eqref{25-6-27-4}, \eqref{2026-5-6-1} and H\"{o}lder's inequality,  we have
\begin{align*}
\sup_{\tau\in[0,t]}\|\rho\|_{N}&\leq e^{C\big(\int_0^t\beta(\tau)\|\mathbf u\|^2_{M+1}d\tau\big)^{\frac12}\big(\int_0^t\beta^{-1}(s)ds\big)^{\frac12}}
\Big(\|\rho_0\|_{N}+C\sup_{\tau\in[0,t]}\|\rho\|_{M}\sqrt t(\int_0^t\|\mathbf u\|_{N+1}^2d\tau)^\frac12\Big)\\
&\leq e^{C\big(\|\rho_0\|^2_M+1\big)^{\frac{1}{2}}\|(\mathbf u_0,p_0,\mathbf h_0)\|_M^{\frac{N-M-r-1}{N-M}}}\Big(\|\rho_0\|_{N}+C\|\rho_0\|_{M}\sqrt t(\int_0^t\|\mathbf u\|_{N+1}^2d\tau)^\frac12\Big),
\end{align*}
which together with \eqref{2026-5-6-4} implies that  \eqref{2026-5-6-2} holds for all $t\in[0,T]$. This completes the proof of Lemma \ref{priori-25-6-27-2}.
\end{proof}

In what follows, we shall establish the highest-order($N-$th order) energy estimates for $(\mathbf u,p,\mathbf h)$, which are necessary to close the \emph{a priori}  assumption \eqref{2026-5-1-1}.
\begin{lemma}\label{priori-combine-2026-6-5}
 Let $M, N\geq3$. Assume $(\rho,\mathbf u,p,\mathbf h)$ is a smooth solution to   the system \eqref{no-heat-mhd3} on $\mathbb T^3\times[0,T]$ for some $T>0$ satisfying \eqref{suppose-1} and \eqref{rhoH-M}. Then there is a positive constant $A_8$ satisfying $A_8\sim\|\rho_0\|^{2(N-1)}_{M}+1$ such that
\begin{align}\label{26-4-11-2}
\frac{d}{dt}&\widetilde{\mathcal E}(t)+\frac{2}{3}\underline\rho\|\mathbf u_t\|^2_{N-1}+\frac{4}{3}\|\mathbf u\|_{N+1}^2\nonumber\\
&\lesssim\|\rho\|^2_{N-1}\|\mathbf u_t\|^2_{2}+(\|\rho_0\|_M\|\mathbf u\|_3\|\mathbf u\|_{N}+\|\rho\|_N\|\mathbf u\|^2_{3})\|\mathbf u\|_N+\|\rho\|^2_{N-1}\|\mathbf u\|^4_3\nonumber\\
&\quad+\|(\mathbf u,p,\mathbf h)\|^2_{3}\|(\mathbf u,p,\mathbf h)\|^2_{N}+(\|\rho_0\|^{2(N-1)}_{M}+1)(1+\|(\mathbf u,p,\mathbf h)\|_2)\|(\mathbf u,p,\mathbf h)\|_2^3\nonumber\\
&\quad+\|(\mathbf u,p,\mathbf h)\|_{3}\|(\mathbf u,p,\mathbf h)\|^2_{N},
\end{align}
and
\begin{equation}\label{26-4-11-1}
\|(\mathbf u,p,\mathbf h)\|^2_N\leq \widetilde{\mathcal E}(t)\leq B_5\big(\|\rho_0\|^{2(N-1)}_{M}+1\big)\|(\mathbf u,p,\mathbf h)\|^2_N
\end{equation}
with $\widetilde{\mathcal E}(t):=A_8\mathcal A(t)+E_{N}(t)$. Here $B_5>0$ is a constant depending only on $\mathbb T^3$, $|\mathbf n|$, $N$, $\mu$, $\lambda$, $\gamma$, $\bar P$, $\underline{\rho}^{-1}$, $\bar\rho$, $M_0$, $r$ and $c$.
\end{lemma}
\begin{proof}
The proof follows the same lines as that of Lemma \ref{priori-combine}. Choosing $k=N$ in Lemma \ref{priori-25-6-27-1} and using \eqref{rhoH-M} and $\|\mathbf u_t\|_{L^\infty}\lesssim\|\mathbf u_t\|_{2}$, we obtain
\begin{equation}\label{26-4-11-1BBBB}
\begin{split}
&\frac{d}{dt} E_{N}(t)+\frac43\|\nabla\Lambda^N\mathbf u\|_0^2+\frac43\|\sqrt{\rho}\Lambda^{N-1}\mathbf u_t\|^2_{0}\\
&\leq C_4\|\rho_0\|^{2(N-1)}_{M}\|\sqrt\rho\mathbf u_t\|^2_0+C_4\|\rho\|^2_{N-1}\|\mathbf u_t\|^2_{2}+C_4\big(\|\rho\|_3\|\mathbf u\|_3\|\mathbf u\|_{N}+\|\rho\|_N\|\mathbf u\|^2_{3}\big)\|\mathbf u\|_N\\
&\quad+C_4\|\rho\|^2_{N-1}\|\mathbf u\|^4_3+C_4\|(\mathbf u,p,\mathbf h)\|^2_{3}\|(\mathbf u,p,\mathbf h)\|^2_{N}+C_4\|(\mathbf u,p,\mathbf h)\|_{3}\|(\mathbf u,p,\mathbf h)\|^2_{N},
\end{split}
\end{equation}
where the positive constant $C_4$ depends only on $\mathbb T^3,|\mathbf n|,N,\mu,\lambda,\gamma,\bar P,\underline{\rho}^{-1},\bar\rho$. By choosing $A_8\geq 2C_4\|\rho_0\|^{2(N-1)}_{M}+1,$  multiplying  \eqref{basic-3} by $A_8$,   adding it to \eqref{26-4-11-1BBBB} and then  employing \eqref{suppose-1},  we have
\begin{align*}
&\frac{d}{dt}\big(A_8\mathcal A(t)+E_{N}(t)\big)+\frac{1}{4}A_8\underline\rho\|\mathbf u_t\|_0^2+\frac{2}{3}\underline\rho\|\Lambda^{N-1}\mathbf u_t\|^2_{0}+A_8\|\nabla\mathbf u\|_0^2+\frac43\|\nabla\Lambda^N\mathbf u\|_0^2\\
&\lesssim\|\rho\|^2_{N-1}\|\mathbf u_t\|^2_{2}+\big(\|\rho_0\|_M\|\mathbf u\|_3\|\mathbf u\|_{N}+\|\rho\|_N\|\mathbf u\|^2_{3}\big)\|\mathbf u\|_N+\|\rho\|^2_{N-1}\|\mathbf u\|^4_3\\
&\quad+\|(\mathbf u,p,\mathbf h)\|^2_{3}\|(\mathbf u,p,\mathbf h)\|^2_{N}+\|(\mathbf u,p,\mathbf h)\|_{3}\|(\mathbf u,p,\mathbf h)\|^2_{N}+\big(1+\|(\mathbf u,p,\mathbf h)\|_2\big)\|(\mathbf u,p,\mathbf h)\|_2^3.
\end{align*}
Arguing as in the derivation of \eqref{2026-6-5-2}, and choosing $A_8$ large enough that $A_8\sim\|\rho_0\|^{2(N-1)}_{M}+1$, we conclude that \eqref{26-4-11-2} holds. Moreover, \eqref{26-4-11-1} follows by the definitions of $\mathcal A$ and $E_N, $ \eqref{basic-3-4-9} and \eqref{26-4-10-2}. This completes the proof of Lemma \ref{priori-combine-2026-6-5}.
\end{proof}
\begin{lemma}\label{priori-2026-5-6-1}
Assume that all the conditions of Lemma \ref{priori-25-6-27-2} hold and that $N>M+2(r+1)$. Then there are two positive constants $\epsilon_2,\epsilon_3<1$, depending only on $\mathbb T^3$, $r$, $c$, $|\mathbf n|$, $M$, $N$, $\mu$, $\lambda$, $\gamma$, $\bar P$, $\underline{\rho}^{-1}$, $\bar\rho$ and $M_0$, if
\begin{align}\label{2026-5-7-1}
&(1+\|\rho_0\|_M)\left(M_\rho\|(\mathbf u_0,p_0,\mathbf h_0)\|^2_M\right)^{\frac{N-M-2r-2}{N-M}}\leq\eta_2,
\end{align}
and
\begin{align}\label{2026-5-9-1}
M_\rho\big(1+\|\rho_0\|^{2(N-1)}_{M}\big)&\|(\mathbf u_0,p_0,\mathbf h_0)\|^2_M+\big(1+\|\rho_0\|^{2(N-1)}_{M}\big)\|(\mathbf u_0,p_0,\mathbf h_0)\|^2_N\nonumber\\
&\quad+M_\rho\|\rho_0\|^2_N\|(\mathbf u_0,p_0,\mathbf h_0)\|_M^{\frac{2(N-M-r-1)}{N-M}}\leq\eta_3,
\end{align}
hold for some $\eta_2\in(0,\epsilon_2]$ and $\eta_3\in(0,\epsilon_3]$, then the following estimate holds:
\begin{align}\label{2026-5-7-2}
\sup_{t\in[0,T]}\|(\mathbf u,p,\mathbf h)\|^2_N+\int_0^T\|\mathbf u_t\|^2_{N-1}dt+\int_0^T\|\mathbf u\|_{N+1}^2dt\leq B_6\eta_3,
\end{align}
where the positive constant $B_6$ depends only on $\mathbb T^3,|\mathbf n|,M,\mu,\lambda,\gamma,\bar P,\underline{\rho}^{-1},\bar\rho,M_0,r,c.$
\end{lemma}
\begin{proof} First,
it follows from \eqref{26-4-11-2} and  Young's inequality that
\begin{equation}\label{2026-5-7-2AAABBB}
\begin{split}
\frac{d}{dt}&\widetilde{\mathcal E}(t)+\frac{2}{3}\underline\rho\|\mathbf u_t\|^2_{N-1}+\frac{4}{3}\|\mathbf u\|_{N+1}^2\\
&\leq C\|\rho\|^2_{N-1}\|\mathbf u_t\|^2_{2}+C_\delta\|\rho\|^2_{N}\|\mathbf u\|^4_M+\delta\|\mathbf u\|^2_N\\&\quad+C\big(1+\|\rho_0\|^{2(N-1)}_{M}\big)\big(1+\|(\mathbf u,p,\mathbf h)\|_M\big)\|(\mathbf u,p,\mathbf h)\|_M^3 \\
&\quad+ C\big(1+\|\rho_0\|_M\big)\big(1+\|(\mathbf u,p,\mathbf h)\|_{M}\big)\|(\mathbf u,p,\mathbf h)\|_{M}\|(\mathbf u,p,\mathbf h)\|^2_{N}.
\end{split}
\end{equation}
Moreover, by choosing $\delta=\frac13$ and using \eqref{26-4-11-1}, \eqref{2026-5-7-2AAABBB},  $\sup_{t\in[0,T]}\|(\mathbf u,p,\mathbf h)\|_{M}<1$ (see \eqref{2026-5-1-1}), we obtain
\begin{equation}\label{2026-5-7-2AAABBBB}
\begin{split}
\frac{d}{dt}&\widetilde{\mathcal E}(t)+\frac{2}{3}\underline\rho\|\mathbf u_t\|^2_{N-1}+\|\mathbf u\|_{N+1}^2\\&\lesssim\|\rho\|^2_{N-1}\|\mathbf u_t\|^2_{2}+\|\rho\|^2_{N}\|\mathbf u\|^4_M+\big(1+\|\rho_0\|_M\big)\|(\mathbf u,p,\mathbf h)\|_{M}\widetilde{\mathcal E}(t)\\
&\quad+\big(1+\|\rho_0\|^{2(N-1)}_{M}\big)\|(\mathbf u,p,\mathbf h)\|_M^3
\end{split}
\end{equation}
which along with   Gronwall's inequality  and   \eqref{25-6-27-3}, \eqref{25-6-27-4},  \eqref{2026-5-6-4}, \eqref{2026-5-6-2}, \eqref{26-4-11-1} and \eqref{2026-5-7-1}, yields
\begin{align}\label{2026-5-7-3}
&\sup_{t\in[0,T]}\widetilde{\mathcal E}(t)+\frac{2}{3}\underline\rho\int_0^T\|\mathbf u_t\|^2_{N-1}dt+\int_0^T\|\mathbf u\|_{N+1}^2dt\nonumber\\
&\leq Ce^{C(1+\|\rho_0\|_M)\int_0^T\|(\mathbf u,p,\mathbf h)\|_{M}dt}\sup_{t\in[0,T]}\Big(\beta^{-1}(t)\|\rho\|^2_{N}\Big)\int_0^T\beta(t)\|\mathbf u_t\|^2_{2}dt\nonumber\\
&\quad+ e^{C(1+\|\rho_0\|_M)\int_0^T\|(\mathbf u,p,\mathbf h)\|_{M}dt}
\Big(C\int_0^T\|\rho\|^2_N\|\mathbf u\|^4_{M}dt\nonumber\\
&\qquad\qquad+\big(1+\|\rho_0\|^{2(N-1)}_{M}\big)\int_0^T\|(\mathbf u,p,\mathbf h)\|_M^3dt
+\widetilde{\mathcal E}(0)\Big).
\end{align}
We  bound term by term above in what follows. Based on \eqref{25-6-27-3},\eqref{2026-5-6-2} and \eqref{2026-5-7-1}, we have
\begin{align*}
&\big(1+\|\rho_0\|_M\big)\int_0^T\|(\mathbf u,p,\mathbf h)\|_{M}dt\\
&\leq(1+\|\rho_0\|_M)\int_0^\infty\Big(\frac{r+1}{N-M}t+\big(B_4M_\rho\|(\mathbf u_0,p_0,\mathbf h_0)\|^2_M\big)^{-\frac{r+1}{N-M}}\Big)^{-\frac{N-M}{2r+2}}dt\\
&=(1+\|\rho_0\|_M)\big(B_4M_\rho\|(\mathbf u_0,p_0,\mathbf h_0)\|^2_M\big)^{\frac{N-M-2r-2}{2(N-M)}}\frac{N-M}{r+1}\int_0^\infty (t+1)^{\frac{M-N}{2r+2}}dt\\
&\leq C(1+\|\rho_0\|_M)\big(M_\rho\|(\mathbf u_0,p_0,\mathbf h_0)\|_M^2\big)^{\frac{N-M-2r-2}{2(N-M)}}
\\&\leq C\eta_2,
\end{align*}
\begin{align*}
\sup_{t\in[0,T]}\big(\beta^{-1}(t)\|\rho\|^2_{N}\big)&\leq A_7\Big(1+\frac{N-M}{r+1}\|(\mathbf u_0,p_0,\mathbf h_0)\|_M^{\frac{-2r-2}{N-M}}\int_0^T\|\mathbf u\|^2_{N+1}dt\Big)\\
&\lesssim\|\rho_0\|^2_N\Big(1+\frac{N-M}{r+1}\|(\mathbf u_0,p_0,\mathbf h_0)\|_M^{\frac{-2r-2}{N-M}}\int_0^T\|\mathbf u\|^2_{N+1}dt\Big),
\end{align*}
\begin{align*}
\int_0^T\big(\|\rho_0\|^{2(N-1)}_{M}+1\big)\|(\mathbf u,p,\mathbf h)\|_M^3dt&\leq\big(1+\|\rho_0\|^{2(N-1)}_{M}\big)\sup_{t\in[0,T]}\|(\mathbf u,p,\mathbf h)\|_M^2\int_0^T\|(\mathbf u,p,\mathbf h)\|_Mdt\\
&\leq CM_\rho\big(\|\rho_0\|^{2(N-1)}_{M}+1\big)\|(\mathbf u_0,p_0,\mathbf h_0)\|^2_M\\&\leq C\eta_3,
\end{align*}
and
\begin{align*}
&\int_0^T\|\rho\|^2_N\|\mathbf u\|^4_{M}dt\\
&\leq\int_0^\infty A_7\big(1+t\int_0^t\|\mathbf u\|_{N+1}^2d\tau \big)\Big(\frac{r+1}{N-M}t+\big(B_4M_\rho\|(\mathbf u_0,p_0,\mathbf h_0)\|^2_M\big)^{-\frac{r+1}{N-M}}\Big)^{-\frac{2N-2M}{r+1}}dt\\
&\leq A_7\big(B_4M_\rho\|(\mathbf u_0,p_0,\mathbf h_0)\|^2_M\big)^{2-\frac{r+1}{N-M}}\frac{N-M}{r+1}\int_0^\infty\big(t+1\big)^{\frac{2(M-N)}{r+1}}dt\\
&\quad+A_7\big(B_4M_\rho\|(\mathbf u_0,p_0,\mathbf h_0)\|^2_M\big)^{2-\frac{2r+2}{N-M}}\Big(\frac{N-M}{r+1}\Big)^{2}\int_0^T\|\mathbf u\|^2_{N+1}dt\int_0^\infty\big(t+1\big)^{\frac{2(M-N)}{r+1}}tdt\\
&\leq C\|\rho_0\|_N^2\Big(M_\rho\|(\mathbf u_0,p_0,\mathbf h_0)\|^2_M\Big)^{2-\frac{r+1}{N-M}}+C\|\rho_0\|_N^2\Big(M_\rho\|(\mathbf u_0,p_0,\mathbf h_0)\|^2_M\Big)^{2-\frac{2r+2}{N-M}}\int_0^T\|\mathbf u\|^2_{N+1}dt\\
&\leq C\eta_2^{\frac{r+1}{N-M}}\eta_3+C\eta_3\int_0^T\|\mathbf u\|^2_{N+1}dt.
\end{align*}
Then inserting the above estimates into \eqref{2026-5-7-3} yields that
\begin{align}\label{2026-5-7-3-3}
&\sup_{t\in[0,T]}\widetilde{\mathcal E}(t)+\frac{2}{3}\underline\rho\int_0^T\|\mathbf u_t\|^2_{N-1}dt+\int_0^T\|\mathbf u\|_{N+1}^2dt\nonumber\\
&\leq C\|\rho_0\|^2_N\Big(1+\|(\mathbf u_0,p_0,\mathbf h_0)\|_M^{\frac{-2r-2}{N-M}}\int_0^T\|\mathbf u\|^2_{N+1}dt\Big)M_\rho\|(\mathbf u_0,p_0,\mathbf h_0)\|^2_M\nonumber\\
&\quad+ C\eta_3+C\eta_3\int_0^T\|\mathbf u\|^2_{N+1}dt+C\big(1+\|\rho_0\|^{2(N-1)}_{M}\big)\|(\mathbf u_0,p_0,\mathbf h_0)\|^2_N\nonumber\\
&\leq C\eta_3+C\eta_3\int_0^T\|\mathbf u\|^2_{N+1}dt.
\end{align}
Hence, by  choosing $\epsilon_3$ small in \eqref{2026-5-9-1} such that $C\epsilon_3\leq\frac14$,  and using  \eqref{2026-5-7-3-3} and \eqref{26-4-11-1},   we finally deduce that  \eqref{2026-5-7-2} holds. This completes the proof of Lemma \ref{priori-2026-5-6-1}.
 \end{proof}
Based on Lemma \ref{priori-4}, Lemma \ref{priori-25-6-27-2} and Lemma \ref{priori-2026-5-6-1}, we have the following proposition.
\begin{proposition}\label{th:main2}
Assume $\mathbf n$ satisfies the Diophantine condition \eqref{Diophantine0} and the initial data \\ $(\rho_0,\mathbf u_0,p_0,\mathbf h_0)\in H^{N}(\mathbb T^3)$ satisfying \eqref{xiajie-1}-\eqref{chuzhixiaoyu0}. Let $M,N$ be two integers satisfying $M\geq r+1$ and $N> M+2(r+1)$. Suppose that $(\rho,\mathbf u,p,\mathbf h)$ is a smooth solution to   the system \eqref{no-heat-mhd3} on $\mathbb T^3\times[0,T_1)$ for some $T_1>0$. Then there is a positive constant $\epsilon$, depending only on $\|\rho_0\|_M$, $\|\rho_0\|_N$, $\mathbb T^3$, $|\mathbf n|$, $M$, $N$, $\mu$, $\lambda$, $\gamma$, $\bar P$, $\underline{\rho}^{-1}$, $\bar\rho$, $M_0$, $r$ and $c$, such that
\begin{align}
&\frac{2}{3}\underline{\rho}\leq\rho\leq\frac32\bar\rho,\ \ \text{in}\ \mathbb T^3\times[0,T_1),\label{26-5-8-1}\\
&\sup_{t\in[0,T_1)}\|\rho\|^2_{M}\leq 2\|\rho_0\|_{M}^2,\label{26-5-8-2}\\
&\sup_{\tau\in[0,t]}\|\rho\|^2_{N}\lesssim\|\rho_0\|^2_{N}\big(1+\epsilon t \big),\label{26-5-8-3}\\
&\sup_{\tau\in[0,t]}\|(\mathbf u,p,\mathbf h)\|^2_M\lesssim \Big(t+\big(M_\rho\|(\mathbf u_0,p_0,\mathbf h_0)\|^2_M\big)^{-\frac{r+1}{N-M}}\Big)^{-\frac{N-M}{r+1}},\label{26-5-8-4}\\
&\sup_{t\in[0,T_1)}\|(\mathbf u,p,\mathbf h)\|^2_N+\int_0^{T_1}\|\mathbf u_t\|^2_{N-1}dt+\int_0^{T_1}\|\mathbf u\|_{N+1}^2dt\leq \epsilon,\label{2026-5-9-9}
\end{align}
for any $t\in[0,T_1), $ as long as
\begin{equation}\label{26-5-8-6}\|(\mathbf u_0,p_0,\mathbf h_0)\|_N\leq\epsilon.\end{equation}
\end{proposition}
\begin{proof}
From \eqref{26-5-8-6} and \eqref{2026-5-1-1}, we deduce by choosing $\epsilon$ sufficiently small that
\begin{align}\label{26-5-8-5}
 &\big(1+\|\rho_0\|_M\big)\|(\mathbf u_0,p_0,\mathbf h_0)\|_{N}^{\frac{2(r+1)}{N-M+r+1}}+ M_\rho^{\frac{N-M+r+1}{N-M}}\|(\mathbf u_0,p_0,\mathbf h_0)\|_{N}^{\frac{2(r+1)}{N-M}}\nonumber\\
 &\leq(1+\|\rho_0\|_M)\epsilon^{\frac{2(r+1)}{N-M+r+1}}+\big(1+\|\rho_0\|^{2(M-1)}_{M}\big)^{\frac{N-M+r+1}{N-M}}\epsilon^{\frac{2(r+1)}{N-M}}\nonumber\\
 &\leq\frac{\epsilon_0}2.
\end{align}
Recalling the definitions of $\epsilon_0$ in Lemma \ref{priori-4} and of $\epsilon_1,\epsilon_2,\epsilon_3$ in Lemmas \ref{priori-25-6-27-2} and \ref{priori-2026-5-6-1}, we set
\begin{align*}T_0:=\sup\Big\{&t\in[0,T_1)\ \Big|\  \frac{1}{2}\underline{\rho}\leq\rho\leq2\bar\rho,\ \ \text{in}\ \mathbb T^3\times[0,t], \sup_{\tau\in[0,t]}\|\rho\|_{M}\leq4\|\rho_0\|_M,\\
& (1+\|\rho_0\|_M)\sup_{\tau\in[0,t]}\|(\mathbf u,p,\mathbf h)\|_{N}^{\frac{2(r+1)}{N-M+r+1}}+ M_\rho^{\frac{N-M+r+1}{N-M}}\sup_{\tau\in[0,t]}\|(\mathbf u,p,\mathbf h)\|_{N}^{\frac{2(r+1)}{N-M}}\leq\epsilon_0\Big\}.
\end{align*}
It follows from \eqref{xiajie-1} and \eqref{26-5-8-5} that $T_0\in(0,T_1].$ By definition of $T_0$, Lemma \ref{priori-4} and \eqref{26-5-8-6}, one has
\begin{align}\label{2026-5-9-2}
\sup_{\tau\in[0,t]}\|(\mathbf u,p,\mathbf h)\|^2_M\lesssim \Big(t+\big(M_\rho\|(\mathbf u_0,p_0,\mathbf h_0)\|^2_M\big)^{-\frac{r+1}{N-M}}\Big)^{-\frac{N-M}{r+1}},
\end{align}
for any $t\in[0,T_0)$.

Recalling \eqref{2026-5-6-4}, \eqref{2026-5-7-1} and \eqref{2026-5-9-1} and using \eqref{26-5-8-6} and $N> M+2(r+1)$, we can again choose $\epsilon$ sufficiently small such that the following estimates hold:
\begin{align*}
&M_\rho^{\frac12}\|(\mathbf u_0,p_0,\mathbf h_0)\|_M^{\frac{N-M-r-1}{N-M}}\leq M_\rho^{\frac12}\epsilon^{\frac{N-M-r-1}{N-M}}\leq\epsilon_1,\\
&(1+\|\rho_0\|_M)\big(M_\rho\|(\mathbf u_0,p_0,\mathbf h_0)\|^2_M\big)^{\frac{N-M-2r-2}{N-M}}\leq(1+\|\rho_0\|_M)\big(M_\rho\epsilon^2\big)^{\frac{N-M-2r-2}{N-M}}\leq\epsilon_2,
\end{align*}
and
\begin{align*}
&M_\rho\big(1+\|\rho_0\|^{2(N-1)}_{M}\big)\|(\mathbf u_0,p_0,\mathbf h_0)\|^2_M+\big(1+\|\rho_0\|^{2(N-1)}_{M}\big)\|(\mathbf u_0,p_0,\mathbf h_0)\|^2_N\nonumber\\
&\quad+M_\rho\|\rho_0\|^2_N\|(\mathbf u_0,p_0,\mathbf h_0)\|_M^{\frac{2(N-M-r-1)}{N-M}}\\
&\leq M_\rho\big(1+\|\rho_0\|^{2(N-1)}_{M}\big)\epsilon^2+\big(1+\|\rho_0\|^{2(N-1)}_{M}\big)\epsilon^2+M_\rho\|\rho_0\|^2_N\epsilon^{\frac{N-M-2r-2}{N-M}}
\epsilon\\&\leq\epsilon_3,
\end{align*}
which together with  the definition of $T_0$,  yields
\begin{equation}\label{2026-5-9-3}
\sup_{t\in[0,T_0)}\|\rho\|^2_{M}\leq 2\|\rho_0\|_{M}^2,
\end{equation}
\begin{equation}\label{2026-5-9-4}
\frac{2}{3}\underline{\rho}\leq\rho\leq\frac32\bar\rho,\ \ \text{in}\ \mathbb T^3\times[0,T_0),
\end{equation}
\begin{align}\label{2026-5-9-5}
\sup_{\tau\in[0,t]}\|\rho\|^2_{N}\lesssim\|\rho_0\|^2_{N}\Big(1+t\int_0^{t}\|\mathbf u\|_{N+1}^2d\tau\Big),
\end{align}
and
\begin{align}\label{2026-5-9-6}
\sup_{t\in[0,T_0)}&\|(\mathbf u,p,\mathbf h)\|^2_N+\int_0^{T_0}\|\mathbf u_t\|^2_{N-1}dt+\int_0^{T_0}\|\mathbf u\|_{N+1}^2dt\nonumber\\
&\leq B_6M_\rho\big(\|\rho_0\|^{2(N-1)}_{M}+1\big)\epsilon^2+B_6\big(1+\|\rho_0\|^{2(N-1)}_{M}\big)\epsilon^2\nonumber\\
&\quad+B_6M_\rho\|\rho_0\|^2_N\epsilon^{\frac{N-M-2r-2}{N-M}}\epsilon\nonumber\\
&\leq\epsilon,
\end{align}
for any $t\in[0,T_0).$  Furthermore, using \eqref{2026-5-9-6} and choosing $\epsilon$ sufficiently small, we can ensure that
\begin{align}\label{2026-5-9-7}&(1+\|\rho_0\|_M)\sup_{t\in[0,T_0)}\|(\mathbf u,p,\mathbf h)\|_{N}^{\frac{2(r+1)}{N-M+r+1}}+ M_\rho^{\frac{N-M+r+1}{N-M}}\sup_{t\in[0,T_0)}\|(\mathbf u,p,\mathbf h)\|_{N}^{\frac{2(r+1)}{N-M}}\nonumber\\
&\leq(1+\|\rho_0\|_M)\epsilon^{\frac{(r+1)}{N-M+r+1}}+ M_\rho^{\frac{N-M+r+1}{N-M}}\epsilon^{\frac{(r+1)}{N-M}}\nonumber\\
&\leq\frac12\epsilon_0.
\end{align}
Thus, it follows from \eqref{2026-5-9-3}, \eqref{2026-5-9-4} and \eqref{2026-5-9-7} that
$T_0=T_1$. Indeed, if $T_0<T_1$, then \eqref{2026-5-9-3}, \eqref{2026-5-9-4} and \eqref{2026-5-9-7}, together with the continuity in time of all the quantities involved, allow the three defining conditions to be propagated to some $T_{2}\in(T_0,T_1]$, namely \begin{equation*}
\sup_{t\in[0,T_2)}\|\rho\|_{M}\leq 4\|\rho_0\|_{M},
\end{equation*}
\begin{equation*}
\frac{1}{2}\underline{\rho}\leq\rho\leq2\bar\rho,\ \ \text{in}\ \mathbb T^3\times[0,T_2),
\end{equation*}
\begin{align*}(1+\|\rho_0\|_M)\sup_{t\in[0,T_2)}\|(\mathbf u,p,\mathbf h)\|_{N}^{\frac{2(r+1)}{N-M+r+1}}+ M_\rho^{\frac{N-M+r+1}{N-M}}\sup_{t\in[0,T_2)}\|(\mathbf u,p,\mathbf h)\|_{N}^{\frac{2(r+1)}{N-M}}\leq\epsilon_0,
\end{align*}
which contradicts the definition of $T_0$ as a supremum. Hence $T_0=T_1$. Then,  employing \eqref{2026-5-9-2}, \eqref{2026-5-9-4}, \eqref{2026-5-9-3} and \eqref{2026-5-9-6}, we deduce that \eqref{26-5-8-1}, \eqref{26-5-8-2}, \eqref{26-5-8-4}  and \eqref{2026-5-9-9} hold.  Moreover, \eqref{26-5-8-3} follows by \eqref{2026-5-9-5} and \eqref{2026-5-9-6}.
This completes the proof of Proposition \ref{th:main2}.
\end{proof}

\section{ Proof of Theorem \ref{th:main1}}\label{th:main-11}
First of all, for any initial data $(\rho_0,\mathbf u_0,p_0,\mathbf h_0)\in H^N(\mathbb T^3)$ with $\rho_0$ bounded away from vacuum, the system \eqref{no-heat-mhd3} possesses a unique local solution; this follows from a standard contraction mapping argument (see, e.g., \cite{KawashimaS,MM}). Thus, given initial data $(\rho_0,\mathbf u_0,p_0,\mathbf h_0)\in H^N(\mathbb T^3)$ satisfying \eqref{xiajie-1}--\eqref{chuzhixiaoyu0}, there exists $T>0$ such that the system \eqref{no-heat-mhd3} has a unique solution $(\rho,\mathbf u,p,\mathbf h)\in C([0,T];H^N(\mathbb T^3))$. Iterating the local well-posedness result extends this solution up to a maximal time of existence $T_{\mathrm{max}}$, and the same local theory provides the continuation criterion: if $T_{\mathrm{max}}<\infty$, then
\begin{equation}\label{2026-5-9-10}
\varlimsup_{t\rightarrow T_{\mathrm{max}}}\Big(\|(\rho,\mathbf u,p,\mathbf h)\|_{N}^2+\|\rho^{-1}\|_{L^\infty}\Big)=\infty.
\end{equation}
Note that the second term is present because the local theory requires the density to stay away from vacuum.

 We now claim that $T_{\mathrm{max}}=\infty$. Indeed, if $T_{\mathrm{max}}<\infty$, then Proposition \ref{th:main2} yields
\begin{align*}
&\sup_{t\in[0,T_{\mathrm{max}})}\|\rho\|^2_{N}\lesssim\|\rho_0\|^2_{N}\big(1+\epsilon T_{\mathrm{max}} \big)<\infty,\\
&\sup_{t\in[0,T_{\mathrm{max}})}\|(\mathbf u,p,\mathbf h)\|^2_N\leq \epsilon<\infty,\\
&\sup_{t\in[0,T_{\mathrm{max}})}\|\rho^{-1}\|_{L^\infty}\leq\frac{3}{2\underline\rho}<\infty,
\end{align*}
where the last bound is the lower bound in \eqref{26-5-8-1}. This contradicts \eqref{2026-5-9-10}, and therefore $T_{\mathrm{max}}=\infty$. This completes the proof of Theorem \ref{th:main1}.
\bigskip
\section*{Acknowledgement}


\vskip .1in
\noindent{\bf Competing interests  }\

On behalf of all authors, the corresponding author states that there is no potential conflicts of interest with respect to the research of this article.
\vskip .1in
\noindent{\bf Authors' contributions   }\

Liening  Qiao, Juntao  Sun, Jiahong Wu and  Fuyi Xu  contributed equally to this work.
\vskip .1in
\noindent{\bf Funding   }\

Qiao and Xu were partially supported by   the
National Natural Science Foundation of China 12326430
and the  Natural Science Foundation of Shandong Province ZR2026MS0022. J. Sun was supported by National Natural Science Foundation of China 12371174.
Wu was partially supported by the National Science Foundation of the United
States under DMS 2104682 and DMS 2309748.
\vskip .1in
\noindent{\bf Availability of data and materials  }\

Data and materials  sharing
not applicable to this article as no data and  materials
 were generated or analyzed during the current study.

\end{document}